\documentclass[a4paper]{article}

\usepackage{url}													
\usepackage{hyperref}											
\usepackage{color}

\usepackage{amsmath, amsthm, amssymb}			
\usepackage{stmaryrd}											
\usepackage{dsfont}												
\usepackage{amsfonts}										
\usepackage[arrow, matrix, curve]{xy}

\newtheorem{theorem}{Theorem}[section]		
\newtheorem{lemma}[theorem]{Lemma}
\newtheorem{proposition}[theorem]{Proposition}
\newtheorem{corollary}[theorem]{Corollary}
\newtheorem*{theoremunnum}{Theorem}

\theoremstyle{definition}
\newtheorem{definition}[theorem]{Definition}
\newtheorem{example}[theorem]{Example}
\newtheorem{remark}[theorem]{Remark}

\newcommand{\Natzeroinfty}{\Natural_0\cup\left\{\infty\right\}}
\newcommand{\Natinfty}{\Natural\cup\left\{\infty\right\}}
\newcommand{\Real}{\mathds{R}}
\newcommand{\Natural}{\mathds{N}}

\newcommand{\funcDefTwo}[4]{\left\{\begin{array}{ll}#1 & \mbox{if }#2\\#3 & \mbox{if }#4\end{array}\right.}
\newcommand{\partderxInd}[1]{\frac{\partial}{\partial x_{#1}}}

\DeclareMathOperator{\ev}{ev}

\DeclareMathOperator{\pr}{pr}
\DeclareMathOperator{\Hom}{Hom}

\begin{document}
\title{\textbf{Exponential laws for spaces of differentiable functions on topological groups}}
\author{Natalie Nikitin}
\date{}
\maketitle

\begin{abstract}
Smooth functions $f:G\to E$ from a topological group $G$ to a locally convex space $E$ were considered
by Riss (1953), Boseck, Czichowski and Rudolph (1981), Belti\c{t}\u{a} and Nicolae (2015), and others,
in varying degrees of generality. The space $C^\infty(G,E)$ of such functions carries a natural
topology, the compact-open $C^\infty$-topology. For topological groups $G$ and $H$, we show that
$C^\infty(G\times H,E)\cong C^\infty(G,C^\infty(H,E))$ as a locally convex space, whenever both $G$
and $H$ are metrizable or both $G$ and $H$ are locally compact. Likewise, $C^k(G, C^l(H,E))$ can be identified
with a suitable space of functions on $G\times H$.
\end{abstract}

\section{Introduction}
\hspace{4mm}
Exponential laws of the form $C^\infty(M\times N,E)\cong C^\infty(M,C^\infty(N,E))$ for spaces of vector-valued
smooth functions on manifolds are essential tools in infinite-dimensional calculus and infinite-dimensional
Lie theory (cf. works by Kriegl and Michor \cite{KrieglA-MichorPW1997ConvSetGlobAna}, 
Kriegl, Michor and Rainer \cite{KrieglA-MichorPW-RainerA2015ExpLawTestFuncDiffGr},
Alzaareer and Schmeding \cite{AlzaareerH-SchmedingA2015DiffMappOnProd},
Gl\"{o}ckner \cite{GlH2015RegInfDimLieGrSemireg}, 
Gl\"{o}ckner and Neeb \cite{GlH-NeebK-Hprep}, 
Neeb and Wagemann \cite{NeebK-H-WagemannF2008},
and others). Stimulated by recent research by Belti\c{t}\u{a} and Nicolae \cite{BeltitaD-NicolaeM2014UnivEnvAlg},
we provide exponential laws for function spaces on topological groups.
\vspace{1mm}

\hspace{4mm}
Let $G$ be a topological group, $U\subseteq G$ be an open subset, $f:U\to E$ be a function to a locally convex
space and $\mathfrak{L}(G):=\Hom_{cts}(\Real,G)$ be the set of continuous one-parameter subgroups
$\gamma:\Real\to G$, endowed with the compact-open topology. For $x\in U$ and $\gamma\in\mathfrak{L}(G)$
let us write

\begin{align*}
	D_\gamma f(x):=\lim_{t\to0}\frac{1}{t}(f(x\cdot\gamma(t))-f(x))
\end{align*}

if the limit exists. Following Riss \cite{RissJ1953CalcDiff} and Boseck et al. 
\cite{BoseckH-CzichowskiG-RudolphK-P1981AnalTopGrGenLieTh}, we say that $f$ is $C^k$
(where $k\in\Natzeroinfty$) if $f$ is continuous, the iterated derivatives

\begin{align*}
	d^{(i)}f(x,\gamma_1,\ldots,\gamma_i):=(D_{\gamma_i}\cdots D_{\gamma_1}f)(x)
\end{align*}

exist for all $x\in U$, $i\in\Natural$ with $i\leq k$
and $\gamma_1,\ldots,\gamma_i\in\mathfrak{L}(G)$, and the maps $d^{(i)}f:U\times\mathfrak{L}(G)^i\to E$
so obtained are continuous. We endow the space $C^k(U,E)$ of all $C^k$-maps $f:U\to E$ with the
compact-open $C^k$-topology (recalled in Definition \ref{def:Ck-map-top-gr}). If $G$ and $H$ are
topological groups and $f:G\times H\to E$ is $C^\infty$, then $f^\vee(x):=f(x,\bullet)\in C^\infty(H,E)$
for all $x\in G$. With a view towards universal enveloping algebras,
Belti\c{t}\u{a} and Nicolae \cite{BeltitaD-NicolaeM2014UnivEnvAlg} verified that $f^\vee\in C^\infty(G,C^\infty(H,E))$
and showed that the linear map

\begin{align*}
	\Phi:C^\infty(G\times H,E)\to C^\infty(G,C^\infty(H,E)),\quad f\mapsto f^\vee
\end{align*}

is a topological embedding.
\vspace{1mm}

\hspace{4mm}
Recall that a Hausdorff space $X$ is called a $k_\Real$-space if functions $f:X\to\Real$ are
continuous if and only if $f\big|_K$ is continuous for each compact subset $K\subseteq X$. We obtain
the following criterion for surjectivity of $\Phi$:

\begin{theoremunnum}[\textbf{A}]
Let $U\subseteq G$, $V\subseteq H$ be open subsets of topological groups $G$ and $H$, and
$E$ be a locally convex space. If $U\times V\times\mathfrak{L}(G)^i\times\mathfrak{L}(H)^j$ is a
$k_\Real$-space for all $i,j\in\Natural_0$, then

\begin{align*}
	\Phi:C^\infty(U\times V,E)\to C^\infty(U,C^\infty(V,E)),\quad f\mapsto f^\vee
\end{align*}

is an isomorphism of topological vector spaces.
\end{theoremunnum}

The condition is satisfied, for example, if both $G$ and $H$ are locally compact or both $G$ and $H$ are
metrizable (see Corollary \ref{cor:exp-law-holds-if-metriz-or-lcp-top-gr}).
\vspace{1mm}

\hspace{4mm}
Generalizing the case of open subsets $U$ and $V$ in locally convex spaces treated by
Alzaareer and Schmeding \cite{AlzaareerH-SchmedingA2015DiffMappOnProd} 
and Gl\"{o}ckner and Neeb \cite{GlH-NeebK-Hprep}, we introduce $C^{k,l}$-functions
$f:U\times V\to E$ on open subsets $U\subseteq G$ and $V\subseteq H$ of topological groups
with separate degrees $k,l\in\Natzeroinfty$ of differentiability in the two variables,
and a natural topology on the space $C^{k,l}(U\times V,E)$ of
such maps (see Definition \ref{def:Ckl-map-top-gr} for details). Theorem (A) is a 
consequence of the following result:

\begin{theoremunnum}[\textbf{B}]
Let $U\subseteq G$, $V\subseteq H$ be open subsets of topological groups $G$ and $H$, let
$E$ be a locally convex space and $k,l\in\Natzeroinfty$. If $U\times V\times\mathfrak{L}(G)^i\times\mathfrak{L}(H)^j$ is a
$k_\Real$-space for all $i,j\in\Natural_0$ with $i\leq k$, $j\leq l$, then

\begin{align*}
	\Phi:C^{k,l}(U\times V,E)\to C^k(U,C^l(V,E)),\quad f\mapsto f^\vee
\end{align*}

is an isomorphism of topological vector spaces.
\end{theoremunnum}
\vspace{1mm}

\hspace{4mm}
\textbf{Notation:} All topological spaces are assumed Hausdorff. We call a map $f:X\to Y$ between topological
spaces $X$ and $Y$ a \textit{topological embedding} if $f$ is a homeomorphism onto its image (it is known
that an injective map $f$ is a topological embedding if and only if the topology on $X$ is initial with respect to $f$,
that is, $X$ carries the coarsest topology making $f$ continuous).
\vspace{1mm}

\hspace{4mm}
\textbf{Acknowledgement:} I wish to express my deepest thanks to Prof. Dr. Helge Gl\"ockner
for precious advice and support. I also wish to thank Rafael Dahmen and Gabor Lukacs for very 
helpful comments on $k_\Real$-spaces.

\section{Differentiability of mappings on topological groups}

\begin{definition}\label{def:one-par-subgr-top-gr}
Let $G$ be a topological group, a \textit{one-parameter subgroup} is a group homomorphism
$\gamma:\Real\to G$.
We denote by $\mathfrak{L}(G):=\Hom_{cts}(\Real, G)$ the set of all
continuous one-parameter subgroups, endowed with the compact-open topology.
\end{definition}

\begin{remark}\label{rem:comp-and-prod-one-par-subgr-top-gr}
If $\gamma\in\mathfrak{L}(G)$ and $\phi:G\to H$ is a continuous group homomorphism, then 
$\phi\circ\gamma\in\mathfrak{L}(H)$ and the map $\mathfrak{L}(\phi):\mathfrak{L}(G)\to\mathfrak{L}(H),
\gamma\mapsto\phi\circ\gamma$ is continuous 
(cf. \cite[Appendix A.5]{GlH-NeebK-Hprep},
see also \cite[Appendix B]{GlH2013ExpLawsUltra}).
\vspace{1mm}

Further, for $\psi=(\gamma,\eta)\in C(\Real,G\times H)$ it is easy to see that $\psi\in\mathfrak{L}(G\times H)$
if and only if $\gamma\in\mathfrak{L}(G)$ and $\eta\in\mathfrak{L}(H)$. Moreover, the natural map
$(\mathfrak{L}(\pr_1),\mathfrak{L}(\pr_2)):\mathfrak{L}(G\times H)\to\mathfrak{L}(G)\times\mathfrak{L}(H)$ 
(where $\pr_1:G\times H\to G$, $\pr_2:G\times H\to H$ are
the coordinate projections) is a homeomorphism (cf. \cite[Appendix A.5]{GlH-NeebK-Hprep}, \cite[Appendix B]{GlH2013ExpLawsUltra}).
\end{remark}

\hspace{4mm}
Now, we define the notion of differentiability along one-parameter subgroups of vector-valued functions
on topological groups:

\begin{definition}\label{def:Ck-map-top-gr}
Let $U\subseteq G$ be an open subset of a topological group $G$ and $E$ be a locally convex space. For a map
$f:U\to E$, $x\in U$ and $\gamma\in\mathfrak{L}(G)$ we define

\begin{align*}
	d^{(1)}f(x,\gamma)
	:=df(x,\gamma)
	:=D_\gamma f(x)
	:=\lim_{t\to0}\frac{1}{t}(f(x\cdot\gamma(t))-f(x))
\end{align*}

if the limit exists. 
\vspace{1mm}

We call $f$ a \textit{$C^k$-map} for $k\in\Natural$ if $f$ is continuous and for each $x\in U$, $i\in\Natural$ with $i\leq k$
and $\gamma_1,\ldots,\gamma_i\in\mathfrak{L}(G)$ the iterative derivatives

\begin{align*}
	d^{(i)}f(x,\gamma_1,\ldots,\gamma_i):=(D_{\gamma_i}\cdots D_{\gamma_1}f)(x)
\end{align*}

exist and define continuous maps

\begin{align*}
	d^{(i)}f:U\times\mathfrak{L}(G)^i\to E,\quad(x,\gamma_1,\ldots,\gamma_i)\mapsto (D_{\gamma_i}\cdots D_{\gamma_1}f)(x).
\end{align*}

If $f$ is $C^k$ for each $k\in\Natural$, then we call $f$ a \textit{$C^\infty$-map} or \textit{smooth}. Further, we call continuous
maps $C^0$ and denote $d^{(0)}f:=f$.
\vspace{1mm}

The set of all $C^k$-maps $f:U\to E$ will be denoted by $C^k(U,E)$ and we endow it with the initial topology with respect
to the family $(d^{(i)})_{i\in\Natural_0, i\leq k}$ of maps

\begin{align*}
	d^{(i)}:C^k(U,E)\to C(U\times\mathfrak{L}(G)^i,E)_{c.o},\quad f\mapsto d^{(i)}f
\end{align*}

(where the right-hand side is equipped with the compact-open topology) turning $C^k(U,E)$ into a Hausdorff locally convex space. (This topology is known as the \textit{compact-open $C^k$-topology}.)
\end{definition}

\begin{definition}\label{def:Ckl-map-top-gr}
Let $U\subseteq G$ and $V\subseteq H$ be open subsets of topological groups $G$ and $H$, let
$E$ be a locally convex space. 
For a map $f:U\times V\to E$, $x\in U$, $y\in V$, $\gamma\in\mathfrak{L}(G)$ and $\eta\in\mathfrak{L}(H)$
we define

\begin{align*}
	d^{(1,0)}f(x,y,\gamma)&:=D_{(\gamma,0)}f(x,y)
	:=\lim_{t\to0}\frac{1}{t}(f(x\cdot\gamma(t),y)-f(x,y))
\end{align*}

and

\begin{align*}
	d^{(0,1)}f(x,y,\eta)&:=D_{(0,\eta)}f(x,y)
	:=\lim_{t\to0}\frac{1}{t}(f(x,y\cdot\eta(t))-f(x,y))
\end{align*}

whenever the limits exist.
\vspace{1mm}

We call a continuous map $f:U\times V\to E$ a 
\textit{$C^{k,l}$-map} for $k,l\in\Natzeroinfty$ if the derivatives

\begin{align*}
	d^{(i,j)}f(x,y,\gamma_1,\ldots,\gamma_i,\eta_1,\ldots,\eta_j)
	:=(D_{(\gamma_i,0)}\cdots D_{(\gamma_1,0)}D_{(0,\eta_j)}\cdots D_{(0,\eta_1)}f)(x,y)
\end{align*}

exist for all $x\in U$, $y\in V$, $i,j\in\Natural_0$ with $i\leq k$, $j\leq l$ and 
$\gamma_1,\ldots,\gamma_i\in\mathfrak{L}(G)$, $\eta_1,\ldots,\eta_j\in\mathfrak{L}(H)$,
and define continuous functions

\begin{align*}
	d^{(i,j)}f:U\times V\times\mathfrak{L}(G)^i\times\mathfrak{L}(H)^j&\to E\\
	(x,y,\gamma_1,\ldots,\gamma_i,\eta_1,\ldots,\eta_j)&\mapsto 
	(D_{(\gamma_i,0)}\cdots D_{(\gamma_1,0)}D_{(0,\eta_j)}\cdots D_{(0,\eta_1)}f)(x,y).
\end{align*}

We endow the space $C^{k,l}(U\times V,E)$ of all $C^{k,l}$-functions $f:U\times V\to E$ with the Hausdorff
locally convex initial topology with respect to the family $(d^{(i,j)})_{i,j\in\Natural_0, i\leq k, j\leq l}$ of maps

\begin{align*}
	d^{(i,j)}:C^{k,l}(U\times V,E)\to C(U\times V\times\mathfrak{L}(G)^i\times\mathfrak{L}(H)^j,E)_{c.o},\quad
	f\mapsto d^{(i,j)}f,
\end{align*}

where the right-hand side is equipped with the compact-open
topology. (The so obtained topology on $C^{k,l}(U\times V,E)$ is called the \textit{compact-open $C^{k,l}$-topology.})
\end{definition}

\begin{remark}\label{rem:f-Ck0-C0l-top-gr-top-space}
If $k=0$ or $l=0$, then the definition of $C^{k,l}$-maps $f:U\times V\to E$ also makes sense if
$U$ or $V$, respectively, is any Hausdorff topological space.
All further results for $C^{k,l}$-maps on topological groups carry over to this situation.
\end{remark}

\begin{remark}\label{rem:f-Ck-or-Ckl-then-relations-to-derivatives-top-gr}
Simple computations show that for $k\geq 1$ a map $f:U\to E$ is $C^k$ if and only if $f$ is $C^1$ and
$df:U\times\mathfrak{L}(G)\to E$ is $C^{k-1,0}$, in this case we have
$d^{(i,0)}(df)=d^{(i+1)}f$
for all $i\in\Natural$ with $i\leq k-1$.
\vspace{1mm}

Similarly, we can show that a map $f:U\times V\to E$ is $C^{k,0}$ if and only if $f$ is $C^{1,0}$
and $d^{(1,0)}f:U\times (V\times\mathfrak{L}(G))\to E$ is $C^{k-1,0}$ with differentials
$d^{(i,0)}(d^{(1,0)}f)=d^{(i+1,0)}f$
for all $i$ as above.
\vspace{1mm}

Further, if a map $f:U\times V\to E$ is $C^{k,l}$, then for each $i,j\in\Natural_0$ with $i\leq k$, $j\leq l$
and fixed $\gamma_1,\ldots,\gamma_i\in\mathfrak{L}(G)$, $\eta_1,\ldots,\eta_j\in\mathfrak{L}(H)$ the map

\begin{align*}
	D_{(\gamma_i,0)}\cdots D_{(\gamma_1,0)}D_{(0,\eta_j)}\cdots D_{(0,\eta_1)}f:U\times V\to E
\end{align*}

is $C^{k,l-j}$ if $i=0$, and $C^{k-i,0}$ otherwise.
\end{remark}

\hspace{4mm}
We warn the reader that the full statement of Schwarz' Theorem does not carry over to non-abelian
topological groups; for a function $f:G\to\Real$ and $\gamma,\eta\in\mathfrak{L}(G)$ such that
$D_\gamma f, D_\eta f, D_\gamma D_\eta f:U\to \Real$ are continuous functions it may
happen that $D_\gamma D_\eta f\neq D_\eta D_\gamma f$ (see Example \ref{exam:schwarz-does-not-work-top-gr}).
Nevertheless, we can prove the following restricted version of Schwarz' Theorem for $C^{k,l}$-maps:

\begin{proposition}\label{prop:f-Ckl-then-Deta-Dgamma-eq-Dgamma-Deta-top-gr}
Let $U\subseteq G$ and $V\subseteq H$ be open subsets of topological groups $G$ and $H$, let $E$ be a 
locally convex space and $f:U\times V\to E$ be a $C^{k,l}$-map for some $k,l\in\Natinfty$. Then the derivatives

\begin{align*}
	(D_{(0,\eta_j)}\cdots D_{(0,\eta_1)}D_{(\gamma_i,0)}\cdots D_{(\gamma_1,0)}f)(x,y)
\end{align*}

exist for all $(x,y)\in U\times V$, $i,j\in\Natural$ with $i\leq k$, $j\leq l$ and $\gamma_1,\ldots,\gamma_i\in\mathfrak{L}(G)$,
$\eta_1,\ldots,\eta_j\in\mathfrak{L}(H)$ and we have

\begin{align*}
	(D_{(0,\eta_j)}\cdots D_{(0,\eta_1)}&D_{(\gamma_i,0)}\cdots D_{(\gamma_1,0)}f)(x,y)\\
	&=(D_{(\gamma_i,0)}\cdots D_{(\gamma_1,0)}D_{(0,\eta_j)}\cdots D_{(0,\eta_1)}f)(x,y).
\end{align*}
\end{proposition}

\begin{proof}
First we prove the assertion for $j=1$ by induction on $i$.
\vspace{1mm}

\textit{Induction start:} Let $(x,y)\in U\times V$, $\gamma\in\mathfrak{L}(G)$ and $\eta\in\mathfrak{L}(H)$. For
suitable $\varepsilon,\delta>0$ we define the continuous map

\begin{align*}
	h:]-\varepsilon,\varepsilon[\times[-\delta,\delta[\to E,\quad (s,t)\mapsto f(x\cdot\gamma(s),y\cdot\eta(t)),
\end{align*}

and obtain the partial derivatives of $h$ via

\begin{align*}
	\frac{\partial h}{\partial s}(s,t)
	&=\lim_{r\to0} \frac{1}{r}(h(s+r,t)-h(s,t))\\
	&=\lim_{r\to0} \frac{1}{r}(f(x\cdot\gamma(s)\cdot\gamma(r),y\cdot\eta(t))-f(x\cdot\gamma(s),y\cdot\eta(t)))\\
	&=D_{(\gamma,0)}f(x\cdot\gamma(s),y\cdot\eta(t)),
\end{align*}

and analogously,

\begin{align*}
	\frac{\partial h}{\partial t}(s,t)=D_{(0,\eta)}f(x\cdot\gamma(s),y\cdot\eta(t))
\end{align*}

and

\begin{align*}
	\frac{\partial^2 h}{\partial s\partial t}(s,t)=(D_{(\gamma,0)}D_{(0,\eta)}f)(x\cdot\gamma(s),y\cdot\eta(t)).
\end{align*}

The obtained maps $\frac{\partial h}{\partial s}$, $\frac{\partial h}{\partial t}$ and
$\frac{\partial^2 h}{\partial s\partial t}$ are continuous, hence we apply
\cite[Lemma 1.3.18]{GlH-NeebK-Hprep},
which states that in this case also the partial derivative $\frac{\partial^2 h}{\partial t\partial s}$ exists and
coincides with $\frac{\partial^2 h}{\partial s\partial t}$. Therefore, we have

\begin{align*}
	(D_{(\gamma,0)}D_{(0,\eta)}f)(x,y)
	&=\frac{\partial^2 h}{\partial s\partial t}(0,0)
	=\frac{\partial^2 h}{\partial t\partial s}(0,0)
	=\lim_{r\to0}\frac{1}{r}\left(\frac{\partial h}{\partial s}(0,r)-\frac{\partial h}{\partial s}(0,0)\right)\\
	&=\lim_{r\to0}\frac{1}{r}\left(D_{(\gamma,0)}f(x,y\cdot\eta(r))-D_{(\gamma,0)}f(x,y)\right)\\
	&=(D_{(0,\eta)}D_{(\gamma,0)}f)(x,y).
\end{align*}

Thus the assertion holds for $i=1$.
\vspace{1mm}

\textit{Induction step:} Now, let $2\leq i\leq k$, $(x,y)\in U\times V$, $\gamma_1,\ldots,\gamma_i\in\mathfrak{L}(G)$
and $\eta\in\mathfrak{L}(H)$. Consider the map

\begin{align*}
	g_1:U\times V\to E,\quad (x,y)\mapsto (D_{(\gamma_{i-1},0)}\cdots D_{(\gamma_1,0)}f)(x,y),
\end{align*}

which is $C^{1,0}$ (see Remark \ref{rem:f-Ck-or-Ckl-then-relations-to-derivatives-top-gr}).
Further, $g_1$ is $C^{0,1}$, because

\begin{align*}
	D_{(0,\eta)}g_1(x,y)
	&=(D_{(0,\eta)}D_{(\gamma_{i-1},0)}\cdots D_{(\gamma_1,0)}f)(x,y)\\
	&=(D_{(\gamma_{i-1},0)}\cdots D_{(\gamma_1,0)}D_{(0,\eta)}f)(x,y),
\end{align*}

by the induction hypothesis, and we see that

\begin{align*}
	(D_{(\gamma_i,0)}D_{(0,\eta)}g_1)(x,y)=(D_{(\gamma_i,0)}D_{(\gamma_{i-1},0)}\cdots D_{(\gamma_1,0)}D_{(0,\eta)}f)(x,y),
\end{align*}

whence $g_1$ is $C^{1,1}$. By the induction start, the derivative $(D_{(0,\eta)}D_{(\gamma_i,0)}g_1)(x,y)$ exists
and equals $(D_{(\gamma_i,0)}D_{(0,\eta)}g_1)(x,y)$, thus we get

\begin{align*}
	(D_{(\gamma_i,0)}\cdots D_{(\gamma_1,0)}D_{(0,\eta)}f)(x,y)
	&= (D_{(\gamma_i,0)}D_{(0,\eta)}g_1)(x,y)\\
	&= (D_{(0,\eta)}D_{(\gamma_i,0)}g_1)(x,y)\\
	&= (D_{(0,\eta)}D_{(\gamma_i,0)}\cdots D_{(\gamma_1,0)}f)(x,y).
\end{align*}

Hence the assertion holds for $j=1$.
\vspace{1mm}

Now, let $2\leq j\leq l$, $1\leq i\leq k$, $\gamma_1,\ldots,\gamma_i\in\mathfrak{L}(G)$,
$\eta_1,\ldots,\eta_j\in~\mathfrak{L}(H)$ and $(x,y)\in U\times V$. By Remark \ref{rem:f-Ck-or-Ckl-then-relations-to-derivatives-top-gr},
the map

\begin{align*}
	g_2:U\times V\to E,\quad (x,y)\mapsto (D_{(0,\eta_{j-1})}\cdots D_{(0,\eta_1)}f)(x,y)
\end{align*}

is $C^{k,1}$, whence we have

\begin{align}\label{eq:der-gamma-der-eta-f}
	&(D_{(0,\eta_j)}D_{(\gamma_i,0)}\cdots D_{(\gamma_1,0)}D_{(0,\eta_{j-1})}\cdots D_{(0,\eta_1)}f)(x,y)\nonumber\\
	&= (D_{(0,\eta_j)}D_{(\gamma_i,0)}\cdots D_{(\gamma_1,0)}g_2)(x,y)\nonumber\\
	&= (D_{(\gamma_i,0)}\cdots D_{(\gamma_1,0)}D_{(0,\eta_j)}g_2)(x,y\nonumber)\\
	&= (D_{(\gamma_i,0)}\cdots D_{(\gamma_1,0)}D_{(0,\eta_j)}\cdots D_{(0,\eta_1)}f)(x,y),
\end{align}

using the first part of the proof. But we also have

\begin{align}\label{eq:der-eta-der-gamma-f}
	&(D_{(0,\eta_j)}D_{(\gamma_i,0)}\cdots D_{(\gamma_1,0)}D_{(0,\eta_{j-1})}\cdots D_{(0,\eta_1)}f)(x,y)\nonumber\\
	&=(D_{(0,\eta_j)}D_{(0,\eta_{j-1})}\cdots D_{(0,\eta_1)}D_{(\gamma_i,0)}\cdots D_{(\gamma_1,0)}f)(x,y),
\end{align}

by induction, whence \eqref{eq:der-eta-der-gamma-f} equals \eqref{eq:der-gamma-der-eta-f}, that is

\begin{align*}
	&(D_{(0,\eta_j)}\cdots D_{(0,\eta_1)}D_{(\gamma_i,0)}\cdots D_{(\gamma_1,0)}f)(x,y)\\
	&=(D_{(\gamma_i,0)}\cdots D_{(\gamma_1,0)}D_{(0,\eta_j)}\cdots D_{(0,\eta_1)}f)(x,y),
\end{align*}

and the proof is finished.
\end{proof}

\begin{corollary}\label{cor:f-Ckl-iff-g-Clk-top-gr}
Let $U\subseteq G$ and $V\subseteq H$ be open subsets of topological groups $G$ and $H$, let $E$ be a
locally convex space and $k,l\in\Natzeroinfty$. A map $f:U\times V\to E$ is $C^{k,l}$ if and only if the map

\begin{align*}
	g:V\times U\to E,\quad (y,x)\mapsto f(x,y)
\end{align*}

is $C^{l,k}$. Moreover, we have

\begin{align*}
	d^{(j,i)}g(y,x,\eta_1,\ldots,\eta_j,\gamma_1,\ldots,\gamma_i)=d^{(i,j)}f(x,y,\gamma_1,\ldots,\gamma_i,\eta_1,\ldots,\eta_j)
\end{align*}

for each $x\in U$, $y\in V$, $i,j\in\Natural_0$ with $i\leq k$, $j\leq l$ and $\gamma_1,\ldots,\gamma_i\in\mathfrak{L}(G)$,
$\eta_1,\ldots,\eta_j\in\mathfrak{L}(H)$.
\end{corollary}

\begin{proof}
First, we assume that $l=0$, that is, $f:U\times V\to E$ is $C^{k,0}$. Then for $x\in U$, $y\in V$ and
$\gamma\in\mathfrak{L}(G)$ we have

\begin{align*}
	d^{(1,0)}f(x,y,\gamma)
	&=\lim_{t\to0}\frac{1}{t}(f(x\cdot\gamma(t),y)-f(x,y))\\
	&=\lim_{t\to0}\frac{1}{t}(g(y,x\cdot\gamma(t))-g(y,x))
	=d^{(0,1)}g(y,x,\gamma),
\end{align*}

and similarly we get $d^{(0,i)}g(y,x,\gamma_1,\ldots,\gamma_i)=d^{(i,0)}f(x,y,\gamma_1,\ldots,\gamma_i)$
for each $i\in\Natural$ with $i\leq k$ and $\gamma_1,\ldots,\gamma_i\in\mathfrak{L}(G)$. The obtained
maps $d^{(0,i)}g:V\times U\times\mathfrak{L}(G)^i\to E$ are obviously continuous, hence $g$ is $C^{0,k}$.
The other implication, as well as the case $k=0$, can be proven analogously. 
\vspace{1mm}

If $k,l\geq 1$, then the assertion follows immediately from Proposition
\ref{prop:f-Ckl-then-Deta-Dgamma-eq-Dgamma-Deta-top-gr}.
\end{proof}

\begin{remark}\label{rem:f-Ckl-then-relations-to-derivatives-with-schwarz-top-gr}
Using Remark \ref{rem:f-Ck-or-Ckl-then-relations-to-derivatives-top-gr}
and Corollary \ref{cor:f-Ckl-iff-g-Clk-top-gr},
we can easily show that if $f:U\times V\to E$ is $C^{k,l}$, then
for all $i,j\in\Natural_0$ with $i\leq k$, $j\leq l$ and fixed $\gamma_1,\ldots,\gamma_i\in\mathfrak{L}(G)$, 
$\eta_1,\ldots,\eta_j\in\mathfrak{L}(H)$ the maps

\begin{align*}
	D_{(\gamma_i,0)}\cdots D_{(\gamma_1,0)}D_{(0,\eta_j)}\cdots D_{(0,\eta_1)}f:U\times V\to E
\end{align*}

are $C^{k-i,l-j}$.
\end{remark}

\hspace{4mm}
The following lemma will be useful:

\begin{lemma}\label{lem:f-Ck-or-Ckl-then-cont-lin-comp-f-Ck-or-Ckl-top-gr}
Let $U\subseteq G$, $V\subseteq H$ be open subsets of topological groups $G$ and $H$, let $E$, $F$ be locally convex spaces,
$\lambda:E\to F$ be a continuous and linear map and $k,l\in\Natzeroinfty$. 
\vspace{3mm}

(a) If $f:U\to E$ is a $C^k$-map, then the map $\lambda\circ f:U\to F$ is $C^k$.
\vspace{3mm}
	
(b) If $f:U\times V\to E$ is a $C^{k,l}$-map, then the map $\lambda\circ f:U\times V\to F$ is $C^{k,l}$.
\end{lemma}

\begin{proof}
To prove $(a)$, let $x\in U$, $\gamma\in\mathfrak{L}(G)$ and $t\neq0$ small enough, then we have

\begin{align*}
	\frac{\lambda(f(x\cdot\gamma(t)))-\lambda(f(x))}{t}=\lambda\left(\frac{f(x\cdot\gamma(t))-f(x)}{t}\right)
	\to\lambda(df(x,\gamma)),
\end{align*}

as $t\to0$, because $\lambda$ is assumed linear and continuous. Therefore, the derivative $d(\lambda\circ f)(x,\gamma)$ exists
and we have $d(\lambda\circ f)(x,\gamma)=(\lambda\circ df)(x,\gamma)$.

Proceeding similarly, for each $i\in\Natural$ with $i\leq k$, $\gamma_1,\ldots,\gamma_i\in\mathfrak{L}(G)$ we obtain
$d^{(i)}(\lambda\circ f)(x,\gamma_1,\ldots,\gamma_i)=(\lambda\circ d^{(i)}f)(x,\gamma_1,\ldots,\gamma_i)$.
Since each of the obtained maps $d^{(i)}(\lambda\circ f)=\lambda\circ d^{(i)}f:U\times\mathfrak{L}(G)^i\to F$ 
is continuous, we see that the map $\lambda\circ f$ is $C^k$.
\vspace{1mm}

Analogously, assertion $(b)$ can be proved showing that for each $i,j\in\Natural_0$ with $i\leq k$, $j\leq l$ we have
$d^{(i,j)}(\lambda\circ f)= \lambda\circ d^{(i,j)}f$.
\end{proof}

\hspace{4mm}
Let us introduce the following notation (the analogue for $C^1$-maps is Lemma \ref{lem:f-C1-iff-f[1]-ex-and-cont-top-gr}):

\begin{lemma}\label{lem:f-C10-iff-f[1-0]-ex-and-cont-top-gr-top-space}
Let $U\subseteq G$, $V\subseteq H$ be open subsets of  topological groups $G$ and $H$, let
$E$ be a locally convex space. A continuous map $f:U\times V\to E$ is $C^{1,0}$ if and only
if there exists a continuous map

\begin{align*}
	f^{[1,0]}:U^{[1]}\times V\to E,
\end{align*}

where

\begin{align*}
	U^{[1]}:=\{(x,\gamma,t)\in U\times\mathfrak{L}(G)\times\Real : x\cdot\gamma(t)\in U\},
\end{align*}

such that 

\begin{align*}
	f^{[1,0]}(x,\gamma,t,y)=\frac{1}{t}(f(x\cdot\gamma(t),y)-f(x,y))
\end{align*}

for each $(x,\gamma,t,y)\in U^{[1]}\times V$ with $t\neq 0$. 
\vspace{1mm}

In this case 
we have $d^{(1,0)}f(x,y,\gamma)=f^{[1,0]}(x,\gamma,0,y)$ for all $x\in U$, $y\in V$ 
and $\gamma\in\mathfrak{L}(G)$.
\end{lemma}

\begin{proof}
First, assume that the map $f^{[1,0]}$ exists and is continuous. Then for $x\in U$, $y\in V$, $\gamma\in\mathfrak{L}(G)$
and $t\neq0$ small enough we have

\begin{align*}
	\frac{1}{t}(f(x\cdot\gamma(t),y)-f(x,y))=f^{[1,0]}(x,\gamma,t,y)\to f^{[1,0]}(x,\gamma,0,y)
\end{align*}

as $t\to0$. Hence $d^{(1,0)}f(x,y,\gamma)$ exists and is given by $f^{[1,0]}(x,\gamma,0,y)$, whence the map

\begin{align*}
	d^{(1,0)}f:U\times V\times\mathfrak{L}(G)\to E,\quad (x,y,\gamma)\mapsto f^{[1,0]}(x,\gamma,0,y)
\end{align*}

is continuous. Thus $f$ is $C^{1,0}$.
\vspace{1mm}

Conversely, let $f$ be a $C^{1,0}$-map. Then we define

\begin{align*}
	f^{[1,0]}:U^{[1]}\times V\to E,\quad 
	f^{[1,0]}(x,\gamma,t,y):=\funcDefTwo{\frac{f(x\cdot\gamma(t),y)-f(x,y)}{t}}{t\neq 0}{d^{(1,0)}f(x,y,\gamma)}{t=0.}
\end{align*}

Since $f$ is continuous, the map $f^{[1,0]}$ is continuous at each $(x,\gamma,t,y)$ with $t\neq0$. 
Given $x_0\in U$ and $\gamma_0\in\mathfrak{L}(G)$, we have $(x_0,\gamma_0,0)\in U^{[1]}$; let
$W:=U_{x_0}\times U_{\gamma_0}\times]-\varepsilon,\varepsilon[\subseteq U^{[1]}$ be an open neighborhood
of $(x_0,\gamma_0,0)$ in $U^{[1]}$, where $U_{x_0}\subseteq U$ and $U_{\gamma_0}\subseteq\mathfrak{L}(G)$ are open neighborhoods
of $x_0$ and $\gamma_0$, respectively, and $\varepsilon>0$.
Now, for fixed $(x,\gamma,y)\in U_{x_0}\times U_{\gamma_0}\times V$ we define the continuous curve

\begin{align*}
	h:]-\varepsilon,\varepsilon[\to E,\quad h(t):=f(x\cdot\gamma(t),y).
\end{align*}

Then for  $t\in]-\varepsilon,\varepsilon[$, $s\neq0$
with $t+s\in]-\varepsilon,\varepsilon[$ we have

\begin{align*}
	\frac{h(t+s)-h(t)}{s}
	&=\frac{f(x\cdot\gamma(t+s),y)-f(x\cdot\gamma(t),y)}{s}\\
	&=\frac{f(x\cdot\gamma(t)\cdot\gamma(s),y)-f(x\cdot\gamma(t),y)}{s}\to d^{(1,0)}f(x\cdot\gamma(t),y,\gamma)
\end{align*}

as $s\to 0$. Thus, the derivative $h'(t)$ exists and is given by $d^{(1,0)}f(x\cdot\gamma(t),y,\gamma)$. The so obtained
map $h':]-\varepsilon,\varepsilon[\to E$ is continuous, hence $h$ is a $C^1$-curve (see \cite{GlH-NeebK-Hprep}
for details on $C^1$-curves with values in locally convex spaces and also on weak integrals which we use in the
next step).
We use the Fundamental Theorem of Calculus (\cite[Proposition 1.1.5]{GlH-NeebK-Hprep})
and obtain for $t\neq0$

\begin{align*}
	f^{[1,0]}(x,\gamma,t,y)
	&=\frac{1}{t}(f(x\cdot\gamma(t),y)-f(x,y))
	=\frac{1}{t}(h(t)-h(0))\\
	&= \frac{1}{t}\int_0^t h'(\tau) d\tau
	= \frac{1}{t}\int_0^t d^{(1,0)}f(x\cdot\gamma(\tau),y,\gamma) d\tau\\
	&=  \frac{1}{t}\int_0^1 t d^{(1,0)}f(x\cdot\gamma(tu),y,\gamma) du
	= \int_0^1 d^{(1,0)}f(x\cdot\gamma(tu),y,\gamma) du.
\end{align*}

But if $t=0$, then

\begin{align*}
	\int_0^1 d^{(1,0)}f(x\cdot\gamma(0),y,\gamma) du
	= d^{(1,0)}f(x,y,\gamma)
	= f^{[1,0]}(x,\gamma,0,y),
\end{align*}

hence

\begin{align*}
	f^{[1,0]}(x,\gamma,t,y) = \int_0^1 d^{(1,0)}f(x\cdot\gamma(tu),y,\gamma) du
\end{align*}

for all $(x,\gamma,t,y)\in W\times V$. Since the map

\begin{align*}
	W\times V\times[0,1]\to E,\quad
	(x,\gamma,t,y,u)\mapsto d^{(1,0)}f(x\cdot\gamma(tu),y,\gamma)
\end{align*}

is continuous, also the parameter-dependent integral

\begin{align*}
	W\times V\to E,\quad
	(x,\gamma,t,y)\mapsto f^{[1,0]}(x,\gamma,t,y) =\int_0^1 d^{(1,0)}f(x\cdot\gamma(tu),y,\gamma) du
\end{align*}

is continuous (by \cite[Lemma 1.1.11]{GlH-NeebK-Hprep}),
in particular in $(x_0,\gamma_0,0,y)$. 
Consequently, $f^{[1,0]}$ is continuous.
\end{proof}

\hspace{4mm}
The following two propositions provide a relation between $C^k$- and $C^{k,l}$-maps on products
of topological groups (a version can also be found in 
\cite{BoseckH-CzichowskiG-RudolphK-P1981AnalTopGrGenLieTh}),
in particular, we will conclude that $C^{\infty,\infty}(U\times V,E)\cong C^\infty(U\times V,E)$
as topological vector spaces (Corollary \ref{cor:f-C-infty-infty-iff-f-C-infty-top-gr}).

\begin{proposition}\label{prop:f-Ckk-then-f-Ck-top-gr}
Let $U\subseteq G$, $V\subseteq H$ be open subsets of topological groups $G$ and $H$, let $E$ be a locally convex space
and $k\in\Natzeroinfty$. If $f:U\times V\to E$ is $C^{k,k}$, then $f$ is $C^k$.

Moreover, the inclusion map

\begin{align*}
	\Psi:C^{k,k}(U\times V,E)\to C^k(U\times V,E),\quad f\mapsto f
\end{align*}

is continuous and linear.
\end{proposition}

\begin{proof}
The case $k=0$ is trivial. For $k\geq1$, we show by induction on $i\in\Natural$ with $i\leq k$ that for all $(x,y)\in U\times V$,
$(\gamma_1,\eta_1),\ldots,(\gamma_i,\eta_i)\in\mathfrak{L}(G\times H)$ the derivatives of $f$ are given by
\begin{align}\label{eq:Ck-derivatives-from-Ckk}
	d^{(i)}f((x,y),(\gamma_1,\eta_1),\ldots,&(\gamma_i,\eta_i))\\
	&=\sum_{j=0}^i\sum_{I_{j,i}}
	d^{(j,i-j)}f(x,y,\gamma_{r_1},\ldots,\gamma_{r_j},\eta_{s_1},\ldots,\eta_{s_{i-j}})\nonumber
\end{align}

where $I_{j,i}:=\{r_1,\ldots,r_j\}\dot{\cup}\{s_1,\ldots,s_{i-j}\}=\{1,\ldots,i\}$.
\vspace{1mm}

\textit{Induction start:} Let $(x,y)\in U\times V$ and $(\gamma,\eta)\in\mathfrak{L}(G\times H)$, that
is, $\gamma\in\mathfrak{L}(G)$ and $\eta\in\mathfrak{L}(H)$, see Remark 
\ref{rem:comp-and-prod-one-par-subgr-top-gr}.
For $t\neq0$ small enough we have

\begin{align*}
	&\frac{f((x,y)\cdot(\gamma(t),\eta(t)))-f(x,y)}{t}\\
	&=\frac{f(x\cdot\gamma(t),y\cdot\eta(t))-f(x,y)}{t}\\
	&=\frac{f(x\cdot\gamma(t),y\cdot\eta(t))-f(x\cdot\gamma(t),y)}{t}+\frac{f(x\cdot\gamma(t),y)-f(x,y)}{t}\\
	&=\frac{g(y\cdot\eta(t),x\cdot\gamma(t))-g(y,x\cdot\gamma(t))}{t}+\frac{f(x\cdot\gamma(t),y)-f(x,y)}{t},
\end{align*}

where $g$ is the map $g:V\times U\to E, (y,x)\mapsto f(x,y)$.
By Corollary \ref{cor:f-Ckl-iff-g-Clk-top-gr},
the map $g$ is $C^{1,1}$, whence the map $g^{[1,0]}$ is defined
and continuous, as well as $f^{[1,0]}$ (see Lemma \ref{lem:f-C10-iff-f[1-0]-ex-and-cont-top-gr-top-space}).
Thus we have

\begin{align*}
	&\frac{g(y\cdot\eta(t),x\cdot\gamma(t))-g(y,x\cdot\gamma(t))}{t}+\frac{f(x\cdot\gamma(t),y)-f(x,y)}{t}\\
	&= g^{[1,0]}(y,\eta,t,x\cdot\gamma(t))+f^{[1,0]}(x,\gamma,t,y)\\
	&\to g^{[1,0]}(y,\eta,0,x)+f^{[1,0]}(x,\gamma,0,y)
\end{align*}

as $t\to0$. Therefore, the derivative $df((x,y),(\gamma,\eta))$ exists and is given by

\begin{align*}
	df((x,y),(\gamma,\eta))
	&=g^{[1,0]}(y,\eta,0,x)+f^{[1,0]}(x,\gamma,0,y)\\
	&= d^{(1,0)}g(y,x,\eta)+d^{(1,0)}f(x,y,\gamma)\\
	&=d^{(0,1)}f(x,y,\eta)+d^{(1,0)}f(x,y,\gamma).
\end{align*}

\textit{Induction step:} Now, let $2\leq i\leq k$, $(x,y)\in U\times V$, 
$(\gamma_1,\eta_1),\ldots,(\gamma_i,\eta_i)\in\mathfrak{L}(G\times H)$. Then for $t\neq0$ small enough we have

\begin{align*}
	&\frac{1}{t}\left(d^{(i-1)}f((x\cdot\gamma_i(t),y\cdot\eta_i(t)),(\gamma_1,\eta_1),\ldots,(\gamma_{i-1},\eta_{i-1}))\right.\\
	&\left.\mbox{ }-d^{(i-1)}f((x,y),(\gamma_1,\eta_1),\ldots,(\gamma_{i-1},\eta_{i-1}))\right)\\
	&=\sum_{j=0}^{i-1}\sum_{I_{j,i-1}} \frac{1}{t}
	\left(d^{(j,i-j-1)}f(x\cdot\gamma_i(t),y\cdot\eta_i(t),\gamma_{r_1},\ldots,\gamma_{r_j},\eta_{s_1},\ldots,\eta_{s_{i-j-1}})\right.\\
	&\left.\mbox{ }- d^{(j,i-j-1)}f(x,y,\gamma_{r_1},\ldots,\gamma_{r_j},\eta_{s_1},\ldots,\eta_{s_{i-j-1}})\right)\\
	&=\sum_{j=0}^{i-1}\sum_{I_{j,i-1}} \frac{1}{t}
	\left((D_{(\gamma_{r_j},0)}\cdots D_{(\gamma_{r_1},0)}D_{(0,\eta_{s_{i-j-1}})}\cdots D_{(0,\eta_1)}f)(x\cdot\gamma_i(t),y\cdot\eta_i(t))\right.\\
	&\left.\mbox{ }-(D_{(\gamma_{r_j},0)}\cdots D_{(\gamma_{r_1},0)}D_{(0,\eta_{s_{i-j-1}})}\cdots D_{(0,\eta_1)}f)(x,y)\right).
\end{align*}

Each of the maps

\begin{align*}
	D_{(\gamma_{r_j},0)}\cdots D_{(\gamma_{r_1},0)}D_{(0,\eta_{s_{i-j-1}})}\cdots D_{(0,\eta_1)}f:U\times V\to E
\end{align*}

is $C^{1,1}$ (see Remark \ref{rem:f-Ckl-then-relations-to-derivatives-with-schwarz-top-gr}),
hence $C^1$ and we have

\begin{align*}
	&\sum_{j=0}^{i-1}\sum_{I_{j,i-1}} \frac{1}{t}
	\left((D_{(\gamma_{r_j},0)}\cdots D_{(\gamma_{r_1},0)}D_{(0,\eta_{s_{i-j-1}})}\cdots D_{(0,\eta_1)}f)(x\cdot\gamma_i(t),y\cdot\eta_i(t))\right.\\
	&\left.\mbox{ }-(D_{(\gamma_{r_j},0)}\cdots D_{(\gamma_{r_1},0)}D_{(0,\eta_{s_{i-j-1}})}\cdots D_{(0,\eta_1)}f)(x,y)\right)\\
	&\to \sum_{j=0}^{i-1}\sum_{I_{j,i-1}}
	\left((D_{(\gamma_{r_j},0)}\cdots D_{(\gamma_{r_1},0)}D_{(0,\eta_i)}D_{(0,\eta_{s_{i-j-1}})}\cdots D_{(0,\eta_1)}f)(x,y)\right.\\
	&\left.\mbox{ }+(D_{(\gamma_i,0)}D_{(\gamma_{r_j},0)}\cdots D_{(\gamma_{r_1},0)}D_{(0,\eta_{s_{i-j-1}})}\cdots D_{(0,\eta_1)}f)(x,y)\right)\\
	&=\sum_{j=0}^{i}\sum_{I_{j,i}}d^{(j,i-j)}f(x,y,\gamma_{r_1},\ldots,\gamma_{r_j},\eta_{s_1},\ldots,\eta_{s_{i-j}})
\end{align*}

as $t\to 0$ (using Proposition \ref{prop:f-Ckl-then-Deta-Dgamma-eq-Dgamma-Deta-top-gr}). Thus \eqref{eq:Ck-derivatives-from-Ckk} holds, and we have

\begin{align*}
	d^{(i)}f=\sum_{j=0}^{i}\sum_{I_{j,i}} d^{(j,i-j)}f\circ g_{I_{j,i}},
\end{align*}

where

\begin{align*}
	g_{I_{j,i}}:U\times V\times\mathfrak{L}(G\times H)^i&\to U\times V\times\mathfrak{L}(G)^j\times\mathfrak{L}(H)^{i-j},\\
	(x,y,(\gamma_1,\eta_1),\ldots,(\gamma_i,\eta_i))&\mapsto(x,y,\gamma_{r_1},\ldots,\gamma_{r_j},\eta_{s_1},\ldots,\eta_{s_{i-j}})
\end{align*}

are continuous maps (see Remark \ref{rem:comp-and-prod-one-par-subgr-top-gr}).
Hence $f$ is $C^k$.
\vspace{1mm}

The linearity of the map $\Psi$ is clear. Further, each of the maps

\begin{align*}
	g_{I_{j,i}}^*:C(U\times V\times\mathfrak{L}(G)^j\times\mathfrak{L}(H)^{i-j},E)_{c.o}&\to C(U\times V\times\mathfrak{L}(G\times H)^i,E)_{c.o},\\
	h\mapsto h\circ g_{I_{j,i}}
\end{align*}

is continuous (see \cite[Appendix A.5]{GlH-NeebK-Hprep}
or \cite[Lemma B.9]{GlH2013ExpLawsUltra}),
whence each of the maps

\begin{align*}
	d^{(i)}\circ\Psi = \sum_{j=0}^{i}\sum_{I_{j,i}} g_{I_{j,i}}^*\circ d^{(j,i-j)}
\end{align*}

is continuous. Since the topology on $C^{k}(U\times V,E)$ is initial with respect to the maps $d^{(i)}$, the continuity of $\Psi$ follows.
\end{proof}

\begin{proposition}\label{prop:f-Ck+l-then-f-Ckl-top-gr}
Let $U\subseteq G$, $V\subseteq H$ be open subsets of topological groups $G$ and $H$, let $E$ be a
locally convex space and $k,l\in\Natural_0$. If $f:U\times V\to E$ is a $C^{k+l}$-map, then $f$ is $C^{k,l}$. 

Moreover, the inclusion map

\begin{align*}
	\Psi:C^{k+l}(U\times V,E)\to C^{k,l}(U\times V,E),\quad f\mapsto f
\end{align*}

is continuous and linear.
\end{proposition}

\begin{proof} We denote by $\varepsilon_G\in\mathfrak{L}(G)$ the constant map $\varepsilon_G:\Real\to G,t\mapsto e_G$,
where $e_G$ is the identity element of $G$, and $\varepsilon_H\in\mathfrak{L}(H)$ is defined analogously.

Let $x\in U$, $y\in V$, $\gamma_1,\ldots,\gamma_i\in\mathfrak{L}(G)$ and $\eta_1,\ldots,\eta_j\in\mathfrak{L}(H)$
for some $i,j\in\Natural_0$ with $i\leq k$, $j\leq l$.
Then we obviously have

\begin{align*}
	d^{(i,j)}f(x,y,&\gamma_1,\ldots,\gamma_i,\eta_1,\ldots,\eta_j)\\
	&=d^{(i+j)}f((x,y),(\gamma_1,\varepsilon_H),\ldots,(\gamma_i,\varepsilon_H),
	(\varepsilon_G,\eta_1),\ldots,(\varepsilon_G,\eta_j)).
\end{align*}

Each of the maps

\begin{align*}
	\rho_{i,j}:U\times V\times \mathfrak{L}(G)^i\times \mathfrak{L}(H)^j&\to U\times V\times \mathfrak{L}(G\times H)^{i+j}\\
	(x,y,\gamma_1,\ldots,\gamma_i,\eta_1,\ldots,\eta_j)&
	\mapsto(x,y,(\gamma_1,\varepsilon_H),\ldots(\gamma_i,\varepsilon_H),(\varepsilon_G,\eta_1),\ldots(\varepsilon_G,\eta_j))
\end{align*}

is continuous (see Remark \ref{rem:comp-and-prod-one-par-subgr-top-gr})
and we have

\begin{align*}
	d^{(i,j)}f=d^{(i+j)}f\circ\rho_{i,j}.
\end{align*}

Therefore, $f$ is $C^{k,l}$. 
\vspace{1mm}

The linearity of the map $\Psi$ is clear. Further, by \cite[Appendix A.5]{GlH-NeebK-Hprep}
(see also \cite[Lemma B.9]{GlH2013ExpLawsUltra}),
each of the maps 

\begin{align*}
	\rho_{i,j}^*:C(U\times V\times \mathfrak{L}(G\times H)^{i+j},E)_{c.o}&\to
	C(U\times V\times\mathfrak{L}(G)^i\times \mathfrak{L}(H)^j,E)_{c.o}\\
	h\mapsto h\circ\rho_{i,j}
\end{align*}

is continuous, whence each of the maps

\begin{align*}
	d^{(i,j)}\circ\Psi = \rho_{i,j}^*\circ d^{(i+j)}
\end{align*}

is continuous. Hence, the continuity of $\Psi$ follows, since the topology on the space $C^{k,l}(U\times V,E)$
is initial with respect to the maps $d^{(i,j)}$.
\end{proof}

\begin{corollary}\label{cor:f-C-infty-infty-iff-f-C-infty-top-gr}
Let $U\subseteq G$, $V\subseteq H$ be open subsets of topological groups $G$ and $H$, let $E$ be a locally convex
space. A map $f:U\times V\to E$ is $C^\infty$ if and only if $f$ is $C^{\infty,\infty}$. 

Moreover, the map

\begin{align*}
	\Psi:C^\infty(U\times V,E)\to C^{\infty,\infty}(U\times V,E),\quad f\mapsto f
\end{align*}

is an isomorphism of topological vector spaces.
\end{corollary}

\begin{proof}
The assertion is an immediate consequence of Propositions \ref{prop:f-Ckk-then-f-Ck-top-gr}
and 
\ref{prop:f-Ck+l-then-f-Ckl-top-gr}.
\end{proof}

\section{The exponential law}

\hspace{4mm}
We recall the classical Exponential Law for spaces of continuous functions, which can be found, for example,
in \cite[Appendix A.5]{GlH-NeebK-Hprep}:

\begin{proposition}\label{prop:classical-exp-law}
Let $X_1$, $X_2$, $Y$ be topological spaces. If $f:X_1\times X_2\to Y$ is a continuous map, then also the map

\begin{equation*}
	f^\vee:X_1\to C(X_2,Y)_{c.o},\quad x\mapsto f^\vee(x):=f(x,\bullet)
\end{equation*}

is continuous. Moreover, the map

\begin{equation*}
	\Phi:C(X_1\times X_2,Y)_{c.o}\to C(X_1,C(X_2,Y))_{c.o},\quad f\mapsto f^\vee
\end{equation*}

is a topological embedding.

If $X_2$ is locally compact or $X_1\times X_2$ is a k-space, or $X_1\times X_2$ is a $k_\Real$-space
and $Y$ is completely regular, then $\Phi$ is a homeomorphism.
\end{proposition}

\hspace{4mm}
The following terminology is used here:

\begin{remark}
(a) A Hausdorff topological space $X$ is called a \textit{$k$-space} if functions $f:X\to Y$ to a topological space $Y$ are continuous
if and only if the restrictions $f\big|_K:K\to Y$ are continuous for all compact subsets $K\subseteq X$. All locally
compact spaces and all metrizable spaces are $k$-spaces.
\vspace{3mm}

(b) A Hausdorff topological space $X$ is called a \textit{$k_\Real$-space} if real-valued functions $f:X\to \Real$ are continuous
if and only if the restrictions $f\big|_K:K\to \Real$ are continuous for all compact subsets $K\subseteq X$. Each
$k$-space is a $k_\Real$-space, hence also each locally compact and each metrizable space is a $k_\Real$-space.
\vspace{3mm}

(c) A Hausdorff topological space $X$ is called \textit{completely regular} if its topology is initial with respect to the
set $C(X,\Real)$. Each Hausdorff locally convex space (moreover, each Hausdorff topological group) is
completely regular, see \cite{HewittE-RossKA1963AbstrHarmAnalysisI}.
\end{remark}

\begin{theorem}\label{thm:exp-law-Ckl-top-emb-top-gr}
Let $U\subseteq G$, $V\subseteq H$ be open subsets of topological groups $G$ and $H$, let $E$ be a locally convex
space and $k,l\in\Natzeroinfty$. Then the following holds:
\vspace{3mm}

(a) If a map $f:U\times V\to E$ is $C^{k,l}$, then the map

	\begin{align*}
		f^\vee(x):=f(x,\bullet):V\to E,\quad y\mapsto f^\vee(x)(y):=f(x,y)
	\end{align*}

	is $C^l$ for each $x\in U$ and the map

	\begin{align*}
		f^\vee:U\to C^l(V,E),\quad x\mapsto f^\vee(x)
	\end{align*}

	is $C^k$.
\vspace{3mm}
	
(b) The map

	\begin{align*}
		\Phi: C^{k,l}(U\times V,E)\to C^k(U,C^l(V,E)),\quad f\mapsto f^\vee
	\end{align*}

	is linear and a topological embedding.

\end{theorem}

\begin{proof}
$(a)$ We will consider the following cases:
\vspace{1mm}

\textit{The case $k=l=0:$} This case is covered by the classical Exponential Law \ref{prop:classical-exp-law}.
\vspace{1mm}

\textit{The case $k=0$, $l\geq1:$} Let $x\in U$; the map $f^\vee(x)=f(x,\bullet)$ is obviously continuous, and for $y\in V$,
$\eta\in\mathfrak{L}(H)$ and $t\neq0$ small enough we have

\begin{align*}
	\frac{1}{t}(f^\vee(x)(y\cdot\eta(t))-f^\vee(x)(y))
	=\frac{1}{t}(f(x,y\cdot\eta(t))-f(x,y))
	\to D_{(0,\eta)}f(x,y)
\end{align*}

as $t\to0$. Thus the derivative $D_\eta(f^\vee(x))(y)$ exists and equals $D_{(0,\eta)}f(x,y)=(D_{(0,\eta)}f)^\vee(x)(y)$. 
Proceeding similarly,
for each $j\in\Natural$ with $j\leq l$ and $\eta_1,\ldots\eta_j\in\mathfrak{L}(H)$, we obtain the derivatives

\begin{align}\label{eq:eta-der-of-f-check}
	\left(D_{\eta_j}\cdots D_{\eta_1}(f^\vee(x))\right)(y)=(D_{(0,\eta_j)}\cdots D_{(0,\eta_1)}f)^\vee(x)(y)
\end{align}

The obtained differentials
$d^{(j)}(f^\vee(x))=(d^{(0,j)}f)^\vee(x):V\times\mathfrak{L}(H)^j\to E$ are continuous,
therefore $f^\vee(x)$ is $C^l$.

Further, by the classical Exponential Law \ref{prop:classical-exp-law},
each of the maps

\begin{align*}
	f^\vee:&U\to C(V,E)_{c.o},\quad x\mapsto f^\vee(x),\\
	(d^{(0,j)}f)^\vee:&U\to C(V\times\mathfrak{L}(H)^j,E)_{c.o},\quad x\mapsto (d^{(0,j)}f)^\vee(x)
\end{align*}

is continuous, and we have $d^{(j)}\circ f^\vee = (d^{(0,j)}f)^\vee$ for all $j\in\Natural_0$ with $j\leq l$. Thus,
the continuity of $f^\vee$ follows from the fact that the topology on $C^l(V,E)$ is initial with respect to 
the maps $d^{(j)}$.
\vspace{1mm}

\textit{The case $k\geq1,l\geq0:$} By the preceding steps, the map $f^\vee(x)$ is $C^l$ for each $x\in U$
(with derivatives given in \eqref{eq:eta-der-of-f-check}). Now we show
by induction on $i\in\Natural$ with $i\leq k$ that

\begin{align}\label{eq:gamma-der-of-f-check}
	(D_{\gamma_i}\cdots D_{\gamma_1}(f^\vee))(x)
	=(D_{(\gamma_i,0)}\cdots D_{(\gamma_1,0)}f)^\vee(x)
\end{align}

for all $x\in U$ and $\gamma_1,\ldots,\gamma_i\in\mathfrak{L}(G)$.
\vspace{1mm}

\textit{Induction start:} Since $f$ is $C^{1,0}$,
by Lemma \ref{lem:f-C10-iff-f[1-0]-ex-and-cont-top-gr-top-space}
the map $f^{[1,0]}:U^{[1]}\times V\to E$ is continuous, hence so is the map $(f^{[1,0]})^\vee:U^{[1]}\to C(V,E)_{c.o}$
(see Proposition \ref{prop:classical-exp-law}).
Let $(x,\gamma,t)\in U^{[1]}$ such that $t\neq 0$ and let $y\in V$, then we have

\begin{align*}
	\frac{1}{t}(f^\vee(x\cdot\gamma(t))(y)-f^\vee(x)(y))
	&=\frac{1}{t}(f(x\cdot\gamma(t),y)-f(x,y))\\
	&= f^{[1,0]}(x,\gamma,t,y)
	=(f^{[1,0]})^\vee(x,\gamma,t)(y).
\end{align*}

Therefore

\begin{align*}
	\frac{1}{t}(f^\vee(x\cdot\gamma(t))-f^\vee(x))
	&= (f^{[1,0]})^\vee(x,\gamma,t)\\
	&\to (f^{[1,0]})^\vee(x,\gamma,0)=(D_{(\gamma,0)}f)^\vee(x)
\end{align*}

as $t\to 0$. Thus, $D_\gamma(f^\vee)(x)$ exists and is given by $(D_{(\gamma,0)}f)^\vee(x)$.
\vspace{1mm}

\textit{Induction step:} Now, let $2\leq i\leq k$, $x\in U$ and $\gamma_1,\ldots,\gamma_i\in\mathfrak{L}(G)$.
For $t\neq0$ small enough we have

\begin{align*}
	&\frac{1}{t}\left((D_{\gamma_{i-1}}\cdots D_{\gamma_1}(f^\vee))(x\cdot\gamma_i(t))
	-(D_{\gamma_{i-1}}\cdots D_{\gamma_1}(f^\vee))(x)\right)\\
	&=\frac{1}{t}\left((D_{(\gamma_{i-1},0)}\cdots D_{(\gamma_1,0)}f)^\vee(x\cdot\gamma_i(t))
	-(D_{(\gamma_{i-1},0)}\cdots D_{(\gamma_1,0)}f)^\vee(x)\right)
\end{align*}

by the induction hypothesis. But the map $D_{(\gamma_{i-1},0)}\cdots D_{(\gamma_1,0)}f:U\times V\to E$
is $C^{1,0}$ (see Remark \ref{rem:f-Ck-or-Ckl-then-relations-to-derivatives-top-gr}),
hence by the induction start
we have 

\begin{align*}
	&\frac{1}{t}\left((D_{(\gamma_{i-1},0)}\cdots D_{(\gamma_1,0)}f)^\vee(x\cdot\gamma_i(t))
	-(D_{(\gamma_{i-1},0)}\cdots D_{(\gamma_1,0)})f^\vee(x)\right)\\
	&\to D_{\gamma_i}((D_{(\gamma_{i-1},0)}\cdots D_{(\gamma_1,0)}f)^\vee)(x)
	=(D_{(\gamma_i,0)}\cdots D_{(\gamma_1,0)}f)^\vee(x),
\end{align*}

which shows that the derivative $(D_{\gamma_i}\cdots D_{\gamma_1}(f^\vee))(x)$ exists
and is given by  $(D_{(\gamma_i,0)}\cdots D_{(\gamma_1,0)}f)^\vee(x)$, thus \eqref{eq:gamma-der-of-f-check}
holds.

From Remark \ref{rem:f-Ckl-then-relations-to-derivatives-with-schwarz-top-gr},
we know that each of the maps

\begin{align*}
	D_{(\gamma_i,0)}\cdots D_{(\gamma_1,0)}f:U\times V\to E
\end{align*}

is $C^{0,l}$, hence
$(D_{(\gamma_i,0)}\cdots D_{(\gamma_1,0)}f)^\vee(x)\in C^l(V,E)$ for each $x\in U$.
Now, it remains to show that each of the maps

\begin{align*}
	d^{(i)}(f^\vee):U\times\mathfrak{L}(G)^i&\to C^l(V,E),\\
	(x,\gamma_1,\ldots,\gamma_i)&\mapsto (D_{\gamma_i}\cdots D_{\gamma_1}(f^\vee))(x)
	=(D_{(\gamma_i,0)}\cdots D_{(\gamma_1,0)}f)^\vee(x)
\end{align*}

is continuous. To this end, let $y\in V$, $j\in\Natural_0$ with $j\leq l$ and $\eta_1,\ldots,\eta_j\in\mathfrak{L}(H)$.
Then we have

\begin{align*}
	&(d^{(j)}\circ d^{(i)}(f^\vee))(x,\gamma_1,\ldots,\gamma_i)(y,\eta_1,\ldots,\eta_j)\\
	&=d^{(j)}(d^{(i)}(f^\vee)(x,\gamma_1,\ldots,\gamma_i))(y,\eta_1,\ldots,\eta_j)\\
	&=\left[D_{\eta_j}\cdots D_{\eta_1}[(D_{\gamma_i}\cdots D_{\gamma_1}(f^\vee))(x)]\right](y)
\end{align*}

Using \eqref{eq:gamma-der-of-f-check} and \eqref{eq:eta-der-of-f-check} in turn we obtain

\begin{align*}
	&\left[D_{\eta_j}\cdots D_{\eta_1}[(D_{\gamma_i}\cdots D_{\gamma_1}(f^\vee))(x)]\right](y)\\
	&=\left[D_{\eta_j}\cdots D_{\eta_1}[(D_{(\gamma_i,0)}\cdots D_{(\gamma_1,0)}f)^\vee(x)]\right](y)\\
	&=(D_{(0,\eta_j)}\cdots D_{(0,\eta_1)}D_{(\gamma_i,0)}\cdots D_{(\gamma_1,0)}f)^\vee(x)(y).
\end{align*}

Finally, from Proposition \ref{prop:f-Ckl-then-Deta-Dgamma-eq-Dgamma-Deta-top-gr}
we conclude

\begin{align*}
	&(D_{(0,\eta_j)}\cdots D_{(0,\eta_1)}D_{(\gamma_i,0)}\cdots D_{(\gamma_1,0)}f)^\vee(x)(y)\\
	&=(D_{(\gamma_i,0)}\cdots D_{(\gamma_1,0)}D_{(0,\eta_j)}\cdots D_{(0,\eta_1)}f)^\vee(x)(y)\\
	&=d^{(i,j)}f(x,y,\gamma_1,\ldots,\gamma_i,\eta_1,\ldots,\eta_j)\\
	&=(d^{(i,j)}f\circ \rho_{i,j})(x,\gamma_1,\ldots,\gamma_i,y,\eta_1,\ldots,\eta_j)\\
	&=(d^{(i,j)}f\circ \rho_{i,j})^\vee(x,\gamma_1,\ldots,\gamma_i)(y,\eta_1,\ldots,\eta_j),
\end{align*}

where each $\rho_{i,j}$ is the continuous map

\begin{align*}
	\rho_{i,j}:U\times\mathfrak{L}(G)^i\times V\times\mathfrak{L}(H)^j
	&\to U\times V\times\mathfrak{L}(G)^i\times\mathfrak{L}(H)^j,\\
	(x,\gamma,y,\eta)&\mapsto (x,y,\gamma,\eta).
\end{align*}

Now, from the classical Exponential Law \ref{prop:classical-exp-law}
follows that the maps

\begin{align*}
	(d^{(i,j)}f\circ \rho_{i,j})^\vee:U\times\mathfrak{L}(G)^i\to C(V\times\mathfrak{L}(H)^j,E)_{c.o}
\end{align*}

are continuous, and we have shown that 

\begin{align}\label{eq:dj-circ-di-f-check-eq-dij-circ-rhoij-check}
	d^{(j)}\circ d^{(i)}(f^\vee)=(d^{(i,j)}f\circ \rho_{i,j})^\vee,
\end{align}

thus the continuity of
$d^{(i)}(f^\vee)$ follows from the fact that the topology on $C^l(V,E)$ is initial with respect to the maps
$d^{(j)}$, whence $f^\vee$ is $C^k$.
\vspace{1mm}

$(b)$
The linearity and injectivity of $\Phi$ is clear. To show that $\Phi$ is a topological embedding we will prove
that the given topology on $C^{k,l}(U\times V,E)$ is initial with respect to $\Phi$.
We define the functions

\begin{align*}
	\rho_{i,j}^*:
	C(U\times V\times\mathfrak{L}(G)^i\times\mathfrak{L}(H)^j,E)_{c.o}&\to 
	C(U\times\mathfrak{L}(G)^i\times V\times\mathfrak{L}(H)^j,E)_{c.o},\\
	g&\mapsto g\circ\rho_{i,j},
\end{align*}

and

\begin{align*}
	\Psi_{i,j}:C(U\times\mathfrak{L}(G)^i\times V\times\mathfrak{L}(H)^j,E)_{c.o}&\to
	C(U\times\mathfrak{L}(G)^i,C(V\times\mathfrak{L}(H)^j,E)_{c.o})_{c.o},\\
	g&\mapsto g^\vee
\end{align*}

for $i,j\in\Natural_0$ such that $i\leq k$, $j\leq l$. Then we have

\begin{align*}
	(d^{(i,j)}f\circ \rho_{i,j})^\vee=(\Psi_{i,j}\circ\rho_{i,j}^*\circ d^{(i,j)})(f).
\end{align*}

On the other hand, we have

\begin{align*}
	d^{(j)}\circ d^{(i)}(f^\vee)=(C(U\times\mathfrak{L}(G)^i,d^{(j)})\circ d^{(i)}\circ\Phi)(f),
\end{align*}

where $(C(U\times\mathfrak{L}(G)^i,d^{(j)})$ are the maps

\begin{align*}
	C(U\times\mathfrak{L}(G)^i,C^l(V,E))_{c.o}&\to
	C(U\times\mathfrak{L}(G)^i,C(V\times\mathfrak{L}(H)^j,E)_{c.o})_{c.o},\\
	g&\mapsto d^{(j)}\circ g.
\end{align*}

Thus, from \eqref{eq:dj-circ-di-f-check-eq-dij-circ-rhoij-check} follows the equality

\begin{align*}
	C(U\times\mathfrak{L}(G)^i,d^{(j)})\circ d^{(i)}\circ\Phi=\Psi_{i,j}\circ\rho_{i,j}^*\circ d^{(i,j)}.
\end{align*}

The maps $d^{(i,j)}$, $\rho_{i,j}^*$ and $\Psi_{i,j}$ are topological embeddings (see definition of the
topology on $C^{k,l}(U\times V,E)$, \cite[Appendix A.5]{GlH-NeebK-Hprep},
and Proposition \ref{prop:classical-exp-law},
respectively), hence by the transitivity of initial topologies \cite[Appendix A.2]{GlH-NeebK-Hprep}
the given topology on $C^{k,l}(U\times V,E)$ is initial with respect to the maps $\Psi_{i,j}\circ\rho_{i,j}^*\circ d^{(i,j)}$.
But by the above equality, this topology is also initial with respect to the maps 
$C(U\times\mathfrak{L}(G)^i,d^{(j)})\circ d^{(i)}\circ\Phi$. Since $d^{(i)}$ and 
$C(U\times\mathfrak{L}(G)^i,d^{(j)})$ are topological embeddings (see definition of the topology
on $C^l(V,E)$ and \cite[Appendix A.5]{GlH-NeebK-Hprep},
respectively) we conclude from \cite[Appendix A.2]{GlH-NeebK-Hprep} that the topology on the
space $C^{k,l}(U\times V,E)$
is initial with respect to $\Phi$. This completes the proof.
\end{proof}

\hspace{4mm}
Now, we go over to the proof of Theorem (B):

\begin{proof}[Proof of Theorem (B)]
We need to show that if $g\in C^k(U,C^l(V,E))$, then the map

\begin{align*}
	g^\wedge:U\times V\to E,\quad g^\wedge(x,y):=g(x)(y)
\end{align*}

(which is continuous, since the locally convex space $E$ is completely regular and we assumed that $U\times V$ is a $k_\Real$-space, see Proposition
\ref{prop:classical-exp-law}).
is $C^{k,l}$. Since $\Phi(g^\wedge)=(g^\wedge)^\vee=g$, the map $\Phi$ will be surjective, hence
a homeomorphism (being a topological embedding by Theorem \ref{thm:exp-law-Ckl-top-emb-top-gr}).
\vspace{1mm}

To this end, we fix $x\in U$, then $g(x)\in C^l(U,E)$ and for $y\in V$, $\eta\in\mathfrak{L}(H)$ and $t\neq0$ small enough
we have

\begin{align*}
	\frac{1}{t}(g^\wedge(x,y\cdot\eta(t))-g^\wedge(x,y))
	=\frac{1}{t}(g(x)(y\cdot\eta(t))-g(x)(y))
	\to d(g(x))(y,\eta)
\end{align*}

as $t\to 0$. Consequently $d^{(0,1)}(g^\wedge)(x,y,\eta)$ exists and equals
$d(g(x))(y,\eta)=(d^{(1)}\circ g)(x)(y,\eta)=(d^{(1)}\circ g)^\wedge(x,y,\eta)$. 
Analogously, for $j\in\Natural_0$ with $j\leq l$ and $\eta_1,\ldots,\eta_j\in\mathfrak{L}(H)$ we obtain the derivatives

\begin{align*}
	d^{(0,j)}(g^\wedge)(x,y,\eta_1,\ldots,\eta_j)
	=(d^{(j)}\circ g)^\wedge(x,y,\eta_1,\ldots,\eta_j).
\end{align*}

But for fixed $(y,\eta_1,\ldots,\eta_j)$ we have

\begin{align*}
	(d^{(j)}\circ g)^\wedge(x,y,\eta_1,\ldots,\eta_j)
	&=(d^{(j)}\circ g)(x)(y,\eta_1,\ldots,\eta_j)\\
	&=(\ev_{(y,\eta_1,\ldots,\eta_j)}\circ\; d^{(j)}\circ g)(x),
\end{align*}

where $\ev_{(y,\eta_1,\ldots,\eta_j)}$ is the continuous linear map

\begin{align*}
	\ev_{(y,\eta_1,\ldots,\eta_j)}:C(V\times\mathfrak{L}(H)^j,E)_{c.o}\to E,\quad
	h\mapsto h(y,\eta_1,\ldots,\eta_j).
\end{align*}

Since also $d^{(j)}:C^l(V,E)\to C(V\times\mathfrak{L}(H)^j,E)_{c.o}$ is continuous and linear, 
the composition $\ev_{(y,\eta_1,\ldots,\eta_j)}\circ\; d^{(j)}\circ g:U\to E$
is $C^k$, by Lemma \ref{lem:f-Ck-or-Ckl-then-cont-lin-comp-f-Ck-or-Ckl-top-gr}.
Thus for $\gamma\in\mathfrak{L}(G)$
and $t\neq0$ small enough we obtain

\begin{align*}
	&\frac{1}{t}(d^{(0,j)}(g^\wedge)(x\cdot\gamma(t),y,\eta_1,\ldots,\eta_j)-d^{(0,j)}(g^\wedge)(x,y,\eta_1,\ldots,\eta_j))\\
	&=\frac{1}{t}((\ev_{(y,\eta_1,\ldots,\eta_j)}\circ\; d^{(j)}\circ g)(x\cdot\gamma(t))-(\ev_{(y,\eta_1,\ldots,\eta_j)}\circ\; d^{(j)}\circ g)(x))\\
	&\to d(\ev_{(y,\eta_1,\ldots,\eta_j)}\circ\; d^{(j)}\circ g)(x,\gamma),
\end{align*}

as $t\to0$. Thus $d^{(1,j)}(g^\wedge)(x,y,\gamma,\eta_1,\ldots,\eta_j)$ is given by

\begin{align*}
	d(\ev_{(y,\eta_1,\ldots,\eta_j)}\circ\; d^{(j)}\circ g)(x,\gamma)
	&=(\ev_{(y,\eta_1,\ldots,\eta_j)}\circ\; d^{(j)}\circ dg)(x,\gamma)\\
	&=(d^{(j)}\circ dg)(x,\gamma)(y,\eta_1,\ldots,\eta_j)\\
	&=(d^{(j)}\circ dg)^\wedge(x,\gamma,y,\eta_1,\ldots,\eta_j).
\end{align*}

Analogously, for each $i\in\Natural_0$ with $i\leq k$ and $\gamma_1,\ldots,\gamma_i\in\mathfrak{L}(G)$ we obtain

\begin{align*}
	d^{(i,j)}(g^\wedge)(x,y,\gamma_1,\ldots,\gamma_i,\eta_1,\ldots,\eta_j)
	=(d^{(j)}\circ d^{(i)}g)^\wedge(x,\gamma_1,\ldots,\gamma_i,y,\eta_1,\ldots,\eta_j).
\end{align*}

To see that $g^\wedge$ is $C^{k,l}$ we need to show that the maps

\begin{align}\label{eq:der-i-j-g-wedge}
	d^{(i,j)}(g^\wedge):U\times V\times\mathfrak{L}(G)^i\times\mathfrak{L}(H)^j&\to E,\\
	(x,y,\gamma_1,\ldots,\gamma_i,\eta_1,\ldots,\eta_j)&\mapsto (d^{(j)}\circ d^{(i)}g)^\wedge(x,\gamma_1,\ldots,\gamma_i,y,\eta_1,\ldots,\eta_j)\nonumber
\end{align}

are continuous for all $i,j\in\Natural_0$ with $i\leq k$, $j\leq l$. To this end, consider the continuous maps 

\begin{align*}
	d^{(j)}\circ d^{(i)}g:U\times\mathfrak{L}(G)^i\to C(V\times\mathfrak{L}(H)^j,E)_{c.o}.
\end{align*}

By Proposition \ref{prop:classical-exp-law},
the maps $(d^{(j)}\circ d^{(i)}g)^\wedge:U\times\mathfrak{L}(G)^i\times V\times\mathfrak{L}(H)^j\to E$
are continuous, since $E$ is completely regular and we assumed that $U\times V\times\mathfrak{L}(G)^i\times\mathfrak{L}(H)^j$ is a 
$k_\Real$-space, hence the maps $d^{(i,j)}(g^\wedge)$ are continuous and $g^\wedge$ is $C^{k,l}$.
\end{proof}

\begin{remark}
Theorem (A) follows from Theorem (B), since $C^{\infty,\infty}(U\times V,E)\cong C^\infty(U\times V,E)$ as a topological
vector space, by Corollary \ref{cor:f-C-infty-infty-iff-f-C-infty-top-gr}.
\end{remark}

\begin{corollary}\label{cor:exp-law-holds-if-metriz-or-lcp-top-gr}
Let $U\subseteq G$, $V\subseteq H$ be open subsets of topological groups $G$ and $H$, let $E$ be
a locally convex space and $k,l\in\Natzeroinfty$. Assume that at least one of the following conditions is satisfied:
\vspace{3mm}

(a) $l=0$ and $V$ is locally compact,
\vspace{3mm}

(b) $k,l<\infty$ and $U\times V\times\mathfrak{L}(G)^k\times\mathfrak{L}(H)^l$ is a $k_\Real$-space,
\vspace{3mm}

(c) $G$ and $H$ are metrizable,
\vspace{3mm}

(d) $G$ and $H$ are locally compact.
\vspace{3mm}

Then the map

\begin{align*}
	\Phi:C^{k,l}(U\times V,E)\to C^k(U,C^l(V,E)),\quad f\mapsto f^\vee
\end{align*}

is a homeomorphism.
\end{corollary}

\begin{proof}
$(a)$ As in the proof of Theorem (B), we need to show that if $g\in C^k(U,C(V,E))$, then 
$g^\wedge\in C^{k,0}(U\times V,E)$. The computations of the derivatives of $g^\wedge$ carry over (with $j=0$),
hence it remains to show that the maps $d^{(i,0)}(g^\wedge)$ in \eqref{eq:der-i-j-g-wedge}
are continuous for all $i\in\Natural_0$ with $i\leq k$. But since $V$ is assumed locally compact, each of the
maps $(d^{(0)}\circ d^{(i)}g)^\wedge:U\times\mathfrak{L}(G)^i\times V\to E$ is continuous by
Proposition \ref{prop:classical-exp-law},
hence so is each of the maps $d^{(i,0)}(g^\wedge)$, as required.
\vspace{1mm}

$(b)$ By \cite[Proposition, p.62]{HusekM1971ProdQuotAndk'Spaces},
if $U\times V\times\mathfrak{L}(G)^k\times\mathfrak{L}(H)^l$
is a $k_\Real$-space, then so is $U\times V\times\mathfrak{L}(G)^i\times\mathfrak{L}(H)^j$ for each $i,j\in\Natural_0$
with $i\leq k$, $j\leq l$. Hence, Theorem (B) holds and $\Phi$ is a homeomorphism.
\vspace{1mm}

$(c)$ Since $G$ is metrizable, the space $C(\Real,G)$ is metrizable (see \cite[Appendix A.5]{GlH-NeebK-Hprep}
or
\cite[Lemma B.21]{GlH2013ExpLawsUltra}), 
whence
so is $\mathfrak{L}(G)\subseteq C(\Real,G)$ as well as $U\times\mathfrak{L}(G)^i$ for each $i\in\Natural_0$, $i\leq k$
as a finite product of metrizable spaces. With a similar argumentation we conclude that also $V\times\mathfrak{L}(H)^j$
is metrizable for each $j\in\Natural_0$ with $j\leq l$, whence so is $U\times V\times\mathfrak{L}(G)^i\times\mathfrak{L}(H)^j$.
But each metrizable space is a $k$-space, hence a $k_\Real$-space. Therefore, Theorem (B)
holds in this case and $\Phi$ is a homeomorphism.

\vspace{1mm}
$(d)$ As $G$ is locally compact, it is known that the identity component $G_0$ of $G$ 
(being a connected locally compact subgroup of $G$)
is a pro-Lie group (in the sense that $G_0$ is complete and every identity neighborhood of $G_0$ contains a
normal subgroup $N$ such that $G/N$ is a Lie group, see 
\cite[Definition 3.25]{HofmannKH-MorrisSA2007LieThConnProLieGr}).
Hence, by 
\cite[Theorem 3.12]{HofmannKH-MorrisSA2007LieThConnProLieGr},
$\mathfrak{L}(G)$ is a pro-Lie algebra,
and from \cite[Proposition 3.7]{HofmannKH-MorrisSA2007LieThConnProLieGr}
follows that $\mathfrak{L}(G)\cong\Real^I$
for some set $I$ as a topological vector space. Since also $H$ is assumed locally compact, for each $i,j\in\Natural_0$
with $i\leq k$, $j\leq l$ we have $U\times V\times\mathfrak{L}(G)^i\times\mathfrak{L}(H)^j\cong U\times V\times (\Real^I)^i
\times(\Real^J)^j$ for some set $J$. Now, from \cite[Theorem 5.6 $(ii)$]{NobleN1970ContOnCartProd}
follows that
$U\times V\times\mathfrak{L}(G)^i\times\mathfrak{L}(H)^j$ is a $k_\Real$-space (being isomorphic to a
product of completely regular locally compact spaces), whence Theorem (B) holds
and $\Phi$ is a homeomorphism.
\end{proof}

\appendix
\section{Some properties of $C^k$- and $C^{k,l}$-functions on topological groups}

\hspace{4mm}
First, we prove a simple chain rule for compositions of continuous group homomorphisms
and $C^k$-functions:

\begin{lemma}\label{lem:f-Ck-then-f-comp-cont-hom-Ck-top-gr}
Let $G$ and $H$ be topological groups, $E$ be a locally convex space. Let $\phi:G\to H$ be a continuous
group homomorphism and $f:V\to E$ be a $C^k$-map ($k\in\Natinfty$) on an open subset $V\subseteq H$.
Then for $U:=\phi^{-1}(V)$ the map

\begin{align*}
	f\circ\phi\big|_U:U\to E,\quad x\mapsto f(\phi(x))
\end{align*}

is $C^k$.
\end{lemma}

\begin{proof}
Obviously, the map $f\circ\phi\big|_U$ is continuous. Now, let $x\in U$ and $\gamma\in\mathfrak{L}(G)$. For
$t\neq0$ small enough we have

\begin{align*}
	\frac{f(\phi(x\cdot\gamma(t)))-f(\phi(x))}{t}=\frac{f(\phi(x)\cdot\phi(\gamma(t)))-f(\phi(x))}{t}\to df(\phi(x),\phi\circ\gamma)
\end{align*}

as $t\to 0$, since $\phi\circ\gamma\in\mathfrak{L}(H)$, see Remark 
\ref{rem:comp-and-prod-one-par-subgr-top-gr}. Therefore 
$d(f\circ\phi\big|_U)(x,\gamma)$ exists and is given by $df(\phi(x),\phi\circ\gamma)$.

Repeating the above steps, we obtain for $i\in\Natural$ with $i\leq k$, $\gamma_1,\ldots,\gamma_i\in\mathfrak{L}(G)$
the derivatives $	d^{(i)}(f\circ\phi\big|_U)(x,\gamma_1,\ldots,\gamma_i)
=d^{(i)}f(\phi(x),\phi\circ\gamma_1,\ldots,\phi\circ\gamma_i)$.

Now, recall that the map $\mathfrak{L}(\phi):\mathfrak{L}(G)\to\mathfrak{L}(H),\eta\mapsto\phi\circ\eta$
is continuous (Remark \ref{rem:comp-and-prod-one-par-subgr-top-gr}), whence
also each of the maps

\begin{align*}
	d^{(i)}(f\circ\phi\big|_U)
	:=(d^{(i)}f)\circ(\phi\big|_U\times\underbrace{\mathfrak{L}(\phi)\times\cdots\times\mathfrak{L}(\phi)}_{\mbox{$i$-times}}):
	U\times\mathfrak{L}(G)^i\to E
\end{align*}

is continuous. Hence $f\circ\phi\big|_U$ is $C^k$.
\end{proof}

\begin{lemma}\label{lem:Ck-and-Ckl-maps-into-prod-top-gr}
Let $U\subseteq G$, $V\subseteq H$ be open subsets of topological groups $G$ and $H$, let
$(E_{\alpha})_{\alpha\in A}$ be a family of locally convex spaces with direct product $E:=\prod_{\alpha\in A}E_{\alpha}$
and the coordinate projections $\pr_{\alpha}:E\to E_{\alpha}$. For $k,l\in\Natzeroinfty$ the
following holds:
\vspace{3mm}

(a) A map $f:U\to E$ is $C^k$ if and only if all of its components $f_{\alpha}:=\pr_{\alpha}\circ f$ are $C^{k}$.
\vspace{3mm}

(b) A map $f:U\times V\to E$ is $C^{k,l}$ if and only if all of its components $f_{\alpha}:=\pr_{\alpha}\circ f$ are $C^{k,l}$.
\end{lemma}

\begin{proof}
To prove $(a)$, first recall that because each of the projections $\pr_{\alpha}$ is continuous and linear, 
the compositions $\pr_{\alpha}\circ f$ are $C^{k}$
if $f$ is $C^{k}$, by Lemma \ref{lem:f-Ck-or-Ckl-then-cont-lin-comp-f-Ck-or-Ckl-top-gr} $(a)$.
\vspace{1mm}

Conversely,  assume that each $f_{\alpha}$ is $C^{k}$ and let $x\in U$, $\gamma\in\mathfrak{L}(G)$
and $t\neq0$ small enough. Then we have

\begin{align*}
	\frac{1}{t}\left(f(x\cdot\gamma(t))-f(x)\right) = \left(\frac{1}{t}\left(f_{\alpha}(x\cdot\gamma(t))-f_{\alpha}(x)\right)\right)_{\alpha\in A}.
\end{align*}

Since $\frac{1}{t}\left(f_{\alpha}(x\cdot\gamma(t))-f_{\alpha}(x)\right)$ converges to $df_\alpha(x,\gamma)$ as $t\to0$
for each $\alpha\in A$, the derivative $df(x,\gamma)$ exists and is given by $\left(df_\alpha(x,\gamma)\right)_{\alpha\in A}$.

Repeating the above steps, we obtain for $i\in\Natural$ with $i\leq k$ and $\gamma_1,\ldots,\gamma_i\in\mathfrak{L}(G)$
the derivatives $d^{(i)}f(x,\gamma_1,\ldots,\gamma_i)=\left(d^{(i)}f_\alpha(x,\gamma_1,\ldots,\gamma_i)\right)_{\alpha\in A}$,
which define continuous maps

\begin{align*}
	d^{(i)}f=\left(d^{(i)}f_\alpha\right)_{\alpha\in A}:U\times\mathfrak{L}(G)^i \to E.
\end{align*}

Therefore, $f$ is $C^k$. 
\vspace{1mm}

The assertion $(b)$ can be proven similarly, by using Lemma 
\ref{lem:f-Ck-or-Ckl-then-cont-lin-comp-f-Ck-or-Ckl-top-gr} $(b)$ and showing that
for all $i,j\in\Natural_0$, with $i\leq k$, $j\leq l$ we have
$d^{(i,j)}f 
= \left(d^{(i,j)}f_{\alpha}\right)_{\alpha\in A}$.
\end{proof}

\hspace{4mm}
The following lemma is a special case of Lemma 
\ref{lem:f-C10-iff-f[1-0]-ex-and-cont-top-gr-top-space}:

\begin{lemma}\label{lem:f-C1-iff-f[1]-ex-and-cont-top-gr}
Let $U\subseteq G$ be an open subset of a topological group $G$, and $E$ be a
locally convex space. A continuous map $f:U\to E$ is $C^1$ if and only if there exists
a continuous map

\begin{align*}
	f^{[1]}:U^{[1]}\to E
\end{align*}

on the open set

\begin{align*}
	U^{[1]}:=\{(x,\gamma,t)\in U\times\mathfrak{L}(G)\times\Real : x\cdot\gamma(t)\in U\}
\end{align*}

such that

\begin{align*}
	f^{[1]}(x,\gamma,t)=\frac{1}{t}(f(x\cdot\gamma(t))-f(x))
\end{align*}

for each $(x,\gamma,t)\in U^{[1]}$ with $t\neq 0$.
\vspace{1mm}

In this case we have $df(x,\gamma)=f^{[1]}(x,\gamma,0)$ for all $x\in U$ and $\gamma\in\mathfrak{L}(G)$.
\end{lemma}

\hspace{4mm}
We use this lemma, as well as the analogue for $C^1$-maps on locally convex spaces (which
can be found in \cite[Lemma 1.2.10]{GlH-NeebK-Hprep}),
for the proof of a chain rule for compositions of $C^k$-functions $f:G\to E$ and $g:E\to F$, which will
be provided after the following version:

\begin{lemma}\label{lem:f-Ck0-gr-sp-g-Ck-lcs-then-g-comp-f-Ck0-top-gr-top-space}
Let $G$ be a topological group, $P$ be a topological space and $E$, $F$ be locally convex spaces. Let $U\subseteq G$, $V\subseteq E$ be open
subsets, and $k\in\Natinfty$. If $f:U\times P\to E$ is a $C^{k,0}$-map such that $f(U\times P)\subseteq V$, and
$g:V\to F$ is a $C^k$-map (in the sense of differentiability on locally convex spaces), then 

\begin{align*}
	g\circ f:U\times P\to F
\end{align*}

is a $C^{k,0}$-map.
\end{lemma}

\begin{proof}
We may assume that $k$ is finite and prove the assertion by induction.
\vspace{1mm}

\textit{Induction start:} Assume that $f$ is $C^{1,0}$, $g$ is $C^1$ and let $x\in U$, $p\in P$ and $\gamma\in\mathfrak{L}(G)$. For $t\neq 0$ small enough
we have

\begin{align*}
	\frac{g(f(x\cdot(t),p))-g(f(x,p))}{t}
	&=\frac{g\left(f(x,p)+t\frac{f(x\cdot\gamma(t),p)-f(x,p)}{t}\right)-g(f(x,p))}{t}\\
	&=\frac{g(f(x,p)+t\cdot f^{[1,0]}(x,\gamma,t,p))-g(f(x,p))}{t}\\
	&=g^{[1]}(f(x,p),f^{[1,0]}(x,\gamma,t,p),t),
\end{align*}

where $g^{[1]}$, $f^{[1,0]}$ are the continuous maps from \cite[Lemma 1.2.10]{GlH-NeebK-Hprep}
and Lemma \ref{lem:f-C10-iff-f[1-0]-ex-and-cont-top-gr-top-space}.
As $t\to0$ we consequently have

\begin{align*}
	\frac{g(f(x\cdot(t),p))-g(f(x,p))}{t}
	&\to g^{[1]}(f(x,p),f^{[1,0]}(x,\gamma,0,p),0)\\
	&=dg(f(x,p),d^{(1,0)}f(x,p,\gamma)).
\end{align*}

Therefore, the derivative $d^{(1,0)}(g\circ f)(x,p,\gamma)$ exists and is given by the directional
derivative $dg(f(x,p),d^{(1,0)}f(x,p,\gamma))$.

Consider the continuous map

\begin{align*}
	h:U\times P\times\mathfrak{L}(G)\to E,\quad (x,p,\gamma)\mapsto f(x,p).
\end{align*}

Since $d^{(1,0)}(g\circ f)(x,p,\gamma) = (dg\circ (h,d^{(1,0)}f))(x,p,\gamma)$, the map

\begin{align*}
	d^{(1,0)}(g\circ f)=dg\circ (h,d^{(1,0)}f):U\times P\times\mathfrak{L}(G)\to F
\end{align*}

is continuous, whence $g\circ f$ is $C^{1,0}$.
\vspace{1mm}

\textit{Induction step:} Now, assume that $f$ is $C^{k,0}$ and $g$ is $C^k$ for some $k\geq2$. By Remark
\ref{rem:f-Ck-or-Ckl-then-relations-to-derivatives-top-gr}, the map $d^{(1,0)}f:U\times(P\times\mathfrak{L}(G))\to E$ is 
$C^{k-1,0}$, and it is easily seen that the map $h:U\times(P\times\mathfrak{L}(G))\to E$ defined in the induction start is $C^{k,0}$.
Hence, using Lemma \ref{lem:Ck-and-Ckl-maps-into-prod-top-gr} $(b)$, we see that
$(h,d^{(1,0)}):U\times(P\times\mathfrak{L}(G))\to E\times E$
is a $C^{k-1,0}$-map. Since $dg:V\times E\to F$ is $C^{k-1}$ (see \cite[Definition 1.3.1]{GlH-NeebK-Hprep}),
the map

\begin{align*}
	d^{(1,0)}(g\circ f)=dg\circ(h,d^{(1,0)}):U\times(P\times\mathfrak{L}(G))\to F
\end{align*}

is $C^{k-1,0}$, by the induction hypothesis, and from Remark \ref{rem:f-Ck-or-Ckl-then-relations-to-derivatives-top-gr}
follows
that $g\circ f$ is $C^{k,0}$.
\end{proof}

\begin{lemma}\label{lem:f-Ck-gr-g-Ck-lcs-then-g-comp-f-Ck-top-gr}
Let $G$ be a topological group, $E$, $F$ be locally convex spaces and $k\in\Natinfty$. Let $U\subseteq G$, $V\subseteq E$ be
open subsets. If $f:U\to E$ is a $C^k$-map with
$f(U)\subseteq V$ and also $g:V\to F$ is a $C^k$-map, then the map

\begin{align*}
	g\circ f:U\to F
\end{align*}

is $C^k$.
\end{lemma}

\begin{proof}
We may assume that $k$ is finite and prove the assertion by induction.
\vspace{1mm}

\textit{Induction start:} Assume that $f$ and $g$ are $C^1$-maps. Analogously to the preceding lemma,
for $x\in U$, $\gamma\in\mathfrak{L}(G)$ and $t\neq0$ small enough we have

\begin{align*}
	\frac{1}{t}(g(f(x\cdot\gamma(t)))-g(f(x)))
	=g^{[1]}(f(x), f^{[1]}(x,\gamma,t),t),
\end{align*}

with continuous maps $f^{[1]}$ as in Lemma \ref{lem:f-C1-iff-f[1]-ex-and-cont-top-gr} and $g^{[1]}$ as in 
\cite[Lemma 1.2.10]{GlH-NeebK-Hprep}.
Thus, the derivative $d(g\circ f)(x,\gamma)$ exists
and we have

\begin{align*}
	d(g\circ f)(x,\gamma)=g^{[1]}(f(x), f^{[1]}(x,\gamma,0),0)=dg(f(x), df(x,\gamma)).
\end{align*}

Using the continuous function

\begin{align*}
	h:U\times\mathfrak{L}(G)\to E,\quad (x,\gamma)\mapsto f(x),
\end{align*}

we see that

\begin{align*}
	d(g\circ f)=dg\circ(h,df):U\times\mathfrak{L}(G)\to F
\end{align*}

is continuous, hence $g\circ f$ is a $C^1$-map.
\vspace{1mm}

\textit{Induction step:} Now, let $f$ and $g$ be $C^k$-maps for some $k\geq 2$. Then the map
$df:U\times\mathfrak{L}(G)\to E$ is $C^{k-1,0}$, by Remark \ref{rem:f-Ck-or-Ckl-then-relations-to-derivatives-top-gr},
and the map $h:U\times\mathfrak{L}(G)\to E$ is obviously $C^{k,0}$. We use Lemma 
\ref{lem:Ck-and-Ckl-maps-into-prod-top-gr} $(b)$
and see that $(h,df):U\times\mathfrak{L}(G)\to E\times E$
is a $C^{k-1,0}$-map. By \cite[Definition 1.3.1]{GlH-NeebK-Hprep},
the map $dg:V\times E\to F$ is $C^{k-1}$, hence
by Lemma \ref{lem:f-Ck0-gr-sp-g-Ck-lcs-then-g-comp-f-Ck0-top-gr-top-space}, the composition

\begin{align*}
	d(g\circ f)=dg\circ(h,df):U\times\mathfrak{L}(G)\to F
\end{align*}

is $C^{k-1,0}$, whence $g\circ f$ is $C^k$, by Remark \ref{rem:f-Ck-or-Ckl-then-relations-to-derivatives-top-gr}.
\end{proof}

\hspace{4mm}
Finally, the following example illustrates that the statement of Schwarz' Theorem does not hold for maps on
non-abelian topological groups.

\begin{example}\label{exam:schwarz-does-not-work-top-gr}
Consider the following subgroup $G$ of $GL_3(\Real)$:

\begin{align*}
	G:=\left\{x=\begin{pmatrix}
	1&x_1&x_2\\
	0&1&x_3\\
	0&0&1
	\end{pmatrix}
	: x_1,x_2,x_3\in\Real
	\right\}
\end{align*}

(known as the Heisenberg group) and $\gamma,\eta\in\mathfrak{L}(G)$ defined as

\begin{align*}
	\gamma(t):=\begin{pmatrix}1&t&0\\0&1&0\\0&0&1\end{pmatrix},\quad
	\eta(t):=\begin{pmatrix}1&0&0\\0&1&t\\0&0&1\end{pmatrix}\quad (\forall t\in\Real).
\end{align*}

Then $G\cong\Real^3$ via

\begin{align*}
	\phi:G\to\Real^3,\quad x:=\begin{pmatrix}1&x_1&x_2\\0&1&x_3\\0&0&1\end{pmatrix}\mapsto(x_1,x_2,x_3).
\end{align*}

Let $g:\Real^3\to\Real$ be a partially $C^2$-map in the usual sense and define

\begin{align*}
	f:=g\circ\phi:G\to\Real.
\end{align*}

Then for each $x\in G$, the derivatives $D_\gamma f(x)$, $D_\eta f(x)$, $(D_\eta D_\gamma f)(x)$ and
$(D_\gamma D_\eta f)(x)$ can be expressed using the partial derivatives of $g$.

First, we have

\begin{align*}
	D_\gamma f(x)
	&=\lim_{t\to0} \frac{1}{t}(f(x\cdot\gamma(t))-f(x))
	=\lim_{t\to0} \frac{1}{t}(g(\phi(x\cdot\gamma(t)))-g(\phi(x)))\\
	&=\lim_{t\to0} \frac{1}{t}(g(x_1+t,x_2,x_3)-g(x_1,x_2,x_3))\\
	&=\lim_{t\to0} \frac{1}{t}(g((x_1,x_2,x_3)+t(1,0,0))-g(x_1,x_2,x_3))
	=\partderxInd{1}g(x_1,x_2,x_3).
\end{align*}

Further,

\begin{align*}
	D_\eta f(x)
	&=\lim_{t\to0} \frac{1}{t}(f(x\cdot\eta(t))-f(x))
	=\lim_{t\to0} \frac{1}{t}(g(\phi(x\cdot\eta(t)))-g(\phi(x)))\\
	&=\lim_{t\to0} \frac{1}{t}(g(x_1,x_2+tx_1,x_3+t)-g(x_1,x_2,x_3))\\
	&=x_1\cdot\partderxInd{2}g(x_1,x_2,x_3)+\partderxInd{3}g(x_1,x_2,x_3).
\end{align*}

Now,

\begin{align*}
	(D_\eta D_\gamma f)(x)
	&=\lim_{t\to0} \frac{1}{t}(D_\gamma f(x\cdot\eta(t))-D_\gamma f(x))\\
	&=\lim_{t\to0} \frac{1}{t}\left(\partderxInd{1}g(x_1,x_2+tx_1,x_3+t)-\partderxInd{1}g(x_1,x_2,x_3)\right)\\
	&=x_1\cdot\frac{\partial^2}{\partial x_1 \partial x_2}g(x_1,x_2,x_3)+\frac{\partial^2}{\partial x_1 \partial x_3}g(x_1,x_2,x_3).
\end{align*}

And, finally

\begin{align*}
	&(D_\gamma D_\eta f)(x)
	=\lim_{t\to0} \frac{1}{t}(D_\eta f(x\cdot\gamma(t))-D_\eta f(x))\\
	&=\lim_{t\to0} \frac{1}{t}\left((x_1+t)\cdot\partderxInd{2}g(x_1+t,x_2,x_3)+\partderxInd{3}g(x_1+t,x_2,x_3)\right.\\
	&\left.\mbox{ }-x_1\cdot\partderxInd{2}g(x_1,x_2,x_3)-\partderxInd{3}g(x_1,x_2,x_3)\right)\\
	&=\lim_{t\to0} \frac{x_1}{t}\left(\partderxInd{2}g(x_1+t,x_2,x_3)-\partderxInd{2}g(x_1,x_2,x_3)\right)\\
	&\mbox{ }+\lim_{t\to0} \frac{1}{t}\left(\partderxInd{3}g(x_1+t,x_2,x_3)-\partderxInd{3}g(x_1,x_2,x_3)\right)
	+\lim_{t\to0} \partderxInd{2}g(x_1+t,x_2,x_3)\\
	&=x_1\cdot\frac{\partial^2}{\partial x_1 \partial x_2}g(x_1,x_2,x_3)+\frac{\partial^2}{\partial x_1 \partial x_3}g(x_1,x_2,x_3)
	+\partderxInd{2}g(x_1,x_2,x_3)\\
	&=(D_\eta D_\gamma f)(x)+\partderxInd{2}g(x_1,x_2,x_3).
\end{align*}

Thus we see that if $\partderxInd{2}g(x_1,x_2,x_3)\neq0$, then $(D_\gamma D_\eta f)(x)\neq(D_\eta D_\gamma f)(x)$.
\end{example}

\end{document}